\documentclass[12pt,a4paper,oneside,reqno,notitlepage]{amsart}

\usepackage{MnSymbol}
\usepackage{mathrsfs}
\usepackage[utf8]{inputenc}
\usepackage{amsmath,amsthm}

\usepackage{tikz-cd}

\newcommand{\ilimit}{\mbox{$\,\displaystyle{\lim_{\longleftarrow}}\,$}}

\usepackage{amsfonts}
\usepackage{amscd,amsmath}
\usepackage{euscript}

\newtheorem{Theorem}{Theorem}[section]
\newtheorem{Theorem-Definition}{Theorem-Definition}[section]
\newtheorem{Corollary}[Theorem]{Corollary}
\newtheorem{Lemma}[Theorem]{Lemma}
\newtheorem{Proposition}[Theorem]{Proposition}
\theoremstyle{definition}
\newtheorem{Remark}[Theorem]{Remark}
\newtheorem{Example}[Theorem]{Example}
\newtheorem{Definition}[Theorem]{Definition}

\usepackage[arrow,matrix,curve]{xy}\SilentMatrices
\def\xyma{\xymatrix@M.7em}

\numberwithin{equation}{section}

\DeclareMathAlphabet\mathbb{U}{msb}{m}{n}
\newcommand{\mono}{\rightarrowtail}
\newcommand{\epi}{\twoheadrightarrow}

\begin{document}

\title{On lengths of $H\mathbb Z$-localization towers}
\author{Sergei O. Ivanov}
\address{Chebyshev Laboratory, St. Petersburg State University, 14th Line, 29b,
Saint Petersburg, 199178 Russia} \email{ivanov.s.o.1986@gmail.com}

\author{Roman Mikhailov}
\address{Chebyshev Laboratory, St. Petersburg State University, 14th Line, 29b,
Saint Petersburg, 199178 Russia and St. Petersburg Department of
Steklov Mathematical Institute} \email{rmikhailov@mail.ru}

\begin{abstract}
In this paper, the $H\mathbb Z$-length of different groups is
studied. By definition, this is the length of $H\mathbb
Z$-localization tower or the length of transfinite lower central
series of $H\mathbb Z$-localization. It is proved that, for a free
noncyclic group, its $H\mathbb Z$-length is
$\geq \omega+2$. For a large class of $\mathbb Z[C]$-modules $M,$ where $C$ is an infinite cyclic group, it is proved that the $H\mathbb Z$-length of the semi-direct product $M\rtimes
C$ is $\leq \omega+1$ and its $H\mathbb Z$-localization can be
described as a central extension of its pro-nilpotent completion.
In particular, this class covers modules $M$, such that $M\rtimes
C$ is finitely presented and $H_2(M\rtimes C)$ is finite.
\end{abstract}
\maketitle
\par\vspace{.25cm}\noindent MSC2010: 55P60, 19C09, 20J06\par\vspace{1cm}
\section{Introduction}
Let $R$ be either a subgring of rationals or a cyclic ring. In his
fundamental work \cite{Bousfield}, A.K. Bousfield introduced the
concept of $HR$-localization. This is a functor in the category of
groups, closely related to the functor of homological localization
of spaces. In this paper we will study the case $R=\mathbb Z$,
that is, $H\mathbb Z$-localization for different groups.

An $H\mathbb Z$-map between two groups is a homomorphism which
induces an isomorphism on $H_1$ and an epimorphism on $H_2$. A
group $\Gamma$ is $H\mathbb Z$-local if any $H\mathbb Z$-map $G\to
H$ induces a bijection ${\sf Mor}(H,\Gamma)\simeq {\sf
Mor}(G,\Gamma)$. Recall that (\cite{Bousfield}, Theorem 3.10) the
class of $H\mathbb Z$-local groups is the smallest class which
contains the trivial group and closed under inverse limits and
central extensions. Given a group $G$, the $H\mathbb
Z$-localization
$$
\eta: G\to EG
$$
can be uniquely characterized by the following two properties:
$\eta$ is an $H\mathbb Z$-map and the group $EG$ is $H\mathbb
Z$-local. These two properties are given as a definition of
$H\mathbb Z$-localization in \cite{Bousfield}. It is shown in
\cite{Bousfield} that, for any $G$, the $H\mathbb Z$-localization
$EG$ exists, unique and transfiniltely nilpotent.

For a group $G$, denote by $\{\gamma_\tau(G)\}$ the transfinite
lower central series of $G$, defined inductively as
$\gamma_{\tau+1}(G):=[\gamma_\tau(G),G]$ and
$\gamma_\alpha=\bigcap_{\tau<\alpha}\gamma_\tau(G)$ for a limit
ordinal $\alpha$. For a $G$, we will call the length of
transfinite lower series of $EG$, i.e. the least ordinal $\tau,$
such that $\gamma_\tau(EG)=1$ by {\it $H\mathbb Z$-length} of $G$
and denote it as $H\mathbb Z\text{-}{\sf length}(G).$

Let $C$ be an infinite cyclic group. A $\mathbb Z[C]$-module $M$
is {\it tame} if and only if $M\rtimes C$ is a finitely presented
group \cite{Bieri-Strebel_81}.
 If $M$ is a tame $C$-module, then ${\sf dim}_{\mathbb Q}(M\otimes \mathbb Q)<\infty$ and there exist a generator $t\in C$ such that the minimal
 polynomial of the linear map
 $t\otimes \mathbb Q$ $:$ $M\otimes \mathbb Q$ $\to$ $M\otimes \mathbb Q$
  is an integral monic polynomial, which is denoted by $\mu_M\in \mathbb Z[x]$
 (see \cite[Theorem C]{Bieri-Strebel_78}
 and Lemma \ref{Lemma_minimal_polynomial}).
We prove the following

\

\noindent{\bf Theorem.} {\it Let $G$ be a metabelian group of the
form $G=M\rtimes C,$ where $M$ is a tame $\mathbb Z[C]$-module and
$\mu_M=(x-\lambda_1)^{m_1}\dots (x-\lambda_l)^{m_l}$ for some
distinct complex numbers $\lambda_1,\dots,\lambda_l$ and $m_i\geq
1.$
\begin{enumerate}
\item Assume that the equality $\lambda_i\lambda_j=1$ holds only
if  $\lambda_i=\lambda_j=1.$ Then $$H\mathbb Z\text{-}{\sf
length}(G)\leq \omega.$$
\item Assume that the equality
$\lambda_i\lambda_j=1$ holds only if either $m_i=m_j=1$ or
$\lambda_i=\lambda_j=1.$ Then $$H\mathbb Z\text{-}{\sf
length}(G)\leq \omega+1.$$
\end{enumerate}
}

\

\noindent As a contrast, we give an example of a finitely
presented metabelian group of the form $M\rtimes C,$ where $M$ is
tame, whose $H\mathbb Z$-length is greater than $\omega+1$. In the
following example, the $\mathbb Z[C]$-module $M$ is tame but
 it does not satisfy the condition of Theorem \ref{Theorem_omega+1}. Let
\begin{equation}\label{rootgroup}
G=\langle a,b,t\ |\ a^t=a^{-1}, b^t=ab^{-1}, [a,b]=1\rangle=\mathbb Z^2 \rtimes C,
\end{equation}
where $C$ acts on $\mathbb Z^2$ by the matrix $\left(\begin{smallmatrix} -1& \ 1 \\ \ 0 & -1 \end{smallmatrix}\right)$.
It is shown in Theorem \ref{theorem6} that the $H\mathbb Z$-length of
$G$ is  $\geq \omega+2$.

Let $M$ be a tame $\mathbb Z[C]$-module and $\mu_M=(x-1)^{m}f,$
for some $m\geq 0$ and $f\in \mathbb Z[x]$ such that $f(1)\ne 0.$
Assume that $f=f_1^{m_1}\dots f_l^{m_l}$ where $f_1,\dots,f_l\in
\mathbb Z[x]$ are distinct irreducible monic polynomials.
 If $f(1)\in \{-1,1\}$, then
$H\mathbb Z\text{-}{\sf length}(M\rtimes C)<\omega.$ (Corollary \ref{perfectm}).

\vspace{.5cm}\noindent{\bf Conjecture.} {\it If $f(1)\notin \{-1,1\},$
then $H\mathbb Z\text{-}{\sf length}(M\rtimes C)\leq\omega+n,$
where
$$n={\rm max}(\{ m_i \mid f_i(0)\in \{ -1,1\} \  \wedge \ f_i(1)\notin \{-1,1\} \}\cup \{0\}).$$
In particular, for any tame $\mathbb Z[C]$-module $M$, $H\mathbb
Z\text{-}{\sf length}(M\rtimes C)<2\omega.$ }

\

It is easy to check that the above theorem together with Corollary \ref{perfectm} and Proposition \ref{Proposition_structure_of_tame_modules} imply the conjecture for $n=0,1.$

For a group $G$, denote by $\hat G$ its pro-nilpotent completion:
$$
\hat G:=\ilimit_n G/\gamma_n(G).
$$
For a finitely generated group $G$, there is a natural isomorphism
(Prop. 3.14 \cite{Bousfield})
$$
EG/\gamma_\omega(EG)=\hat G.
$$
Therefore, for finitely generated groups, $H\mathbb
Z$-localization gives a natural extension of the pro-nilpotent
completion.

The pro-nilpotent completion of a finitely generated group $G$ is
always $H\mathbb Z$-local and the map $G\to \hat G$ induces an
isomorphism on $H_1$. Therefore, for such a group, the following
conditions are equivalent:\\ 1) the natural epimorphism
$EG\twoheadrightarrow \hat G$ is an isomorphism;\\ 2)
$H\mathbb Z$-length of $G$ $\leq \omega$;\\
3) The natural map $H_2(G)\to H_2(\hat G)$ is an epimorphism.

A simple example of a group with $H\mathbb Z$-length $\omega$ is
the following. Let
$$
G=\langle a, t\ |\ a^t=a^3\rangle=\mathbb Z[1/3] \rtimes C
.$$ Here $C$ acts on $\mathbb Z[1/3]$ as the multiplication by $3.$
Then the pro-nilpotent completion has the structure $\hat
G=\mathbb Z_2\rtimes C$, where the cyclic group $C=\langle
t\rangle$ acts on 2-adic integers as the multiplication by $3$.
Looking at the homology spectral sequence for an extension $1\to
\mathbb Z_2\to \hat G\to C\to 1$, we obtain $H_2(\hat
G)=\Lambda^2(\mathbb Z_2)\otimes \mathbb Z/9=0.$ Therefore,
$EG=\hat G$. Since the group $G$ is not pre-nilpotent, $H\mathbb
Z\text{-}{\sf length}(G)=\omega$.

The above example is an exception. In most cases, the description
of $H\mathbb Z$-localization as well as the computation of
$H\mathbb Z$-length for a given group is a difficult problem. It
is shown in \cite{Bousfield} that $H\mathbb Z$-length of the Klein
bottle group $G_{Kl}:=\langle a,t\ |\ a^t=a^{-1}\rangle = \mathbb Z \rtimes C$ is
greater than $\omega$. As a corollary, it is concluded in
\cite{Bousfield} that $H\mathbb Z$-length of any non-cyclic free
group also is greater than $\omega$. Our Theorem
\ref{Theorem_omega+1} implies that $H\mathbb Z\text{-}{\sf
length}(G_{Kl})=\omega+1$ and that the $H\mathbb Z$-localization
$EG_{Kl}$ lives in the central extension
$$
1\to \Lambda^2(\mathbb Z_2)\to EG_{Kl}\to \mathbb Z_2\rtimes
\langle t\rangle\to 1,
$$
where the action of $t$ on 2-adic integers is negation. Moreover, we give a more explicit description of $EG_{Kl}$ in
Proposition \ref{Proposition_localisation_of_Klein_bottle}.  The
$H\mathbb Z$-length of a free non-cyclic group remains a mystery
for us, however, we prove the following

\

\noindent{\bf Theorem.} {\it  Let $F$ be a free group of rank\:   $\geq 2$.
Then
$H\mathbb Z\text{-}{\sf length}(F)\geq \omega+2$.}

\

\noindent Briefly recall the scheme of the proof. Consider an
extension of the group (\ref{rootgroup}) given by presentation
\begin{equation}\label{rootgroup1}
\Gamma:=\langle a,b,t\ |\ [[a,b],a]=[[a,b],b]=1,\ a^t=a^{-1},\
b^t=ab^{-1}\rangle
\end{equation}
We follow the Bousfield scheme of comparison of pro-nilpotent
completions for a free group and the group $\Gamma$. Consider a
free simplicial resolution of $\Gamma$ with $F_0=F$. Group
$\Gamma$ has finite second homology, therefore, $\ilimit^1$ of its
Baer invariants is zero, therefore, $\pi_0$ of the pro-nilpotent
completion of the free simplicial resolution equals to the
pro-nilpotent completion of $\Gamma$. The group $\Gamma$ has
$H\mathbb Z$-length greater than $\omega+1$ and the result about
$H\mathbb Z$-length of $F$ follows from natural properties of the
$H\mathbb Z$-localization tower. Observe that the same method does
not work for the group (\ref{rootgroup}) (as well as for all
groups of the type $M\rtimes C$ for abelian $M$), since
$\ilimit^1$ of Baer invariants of $G$ is huge (this follows from
proposition \ref{h2pr} and an analysis of the tower of Baer
invariants for a metabelian group).

The paper is organized as follows. In section 2, we present the
theory of relative central extensions, which is a generalisation of the standard theory of central extensions. A non-limit step in the construction
of Bousfield's tower can be viewed as the universal relative
extension. Section 2 is technical and introductory, it may be
viewed just as a comment to the section 3 of \cite{Bousfield}. In
section 3 we recall the exact sequences in homology from
\cite{Eckmann}, \cite{Eckmann1} for central stem-extensions.
Observe that the universal relative extensions used for the
construction of $H\mathbb Z$-tower are stem-extensions.
Proposition \ref{alpha33} gives the main trick: for a cyclic group
$C$, $\mathbb Z[C]$-module $M$, and a central stem-extension
$N\hookrightarrow G\twoheadrightarrow M\rtimes C,$ the composite
map $H_3(M\rtimes C)\to (M\rtimes C)\otimes N\twoheadrightarrow N$
can be decomposed as $H_3(M\rtimes C)\to H_2(M)^C\to H_2(M)_C\to
N$. This trick gives a possibility to analyze the homology of the
$\omega+1$-st term of the $H\mathbb Z$-localization tower for
groups of the type $M\rtimes C$.

Using the properties of tame modules we show in section
\ref{tamemodules} that the question about the $H\mathbb Z$-length
of the group $M\rtimes C$ with a tame $\mathbb Z[C]$-module $M$,
can be reduced to the same question for the group $(M/N)\rtimes
C$, where $N$ is the largest nilpotent submodule of $M$. It is
shown in \cite{IvanovMikhailov} that, for a finitely presented
metabelian group $G$, the cokernel (denoted by
$H_2(\eta_\omega)(G)$) of the natural map $H_2(G)\to H_2(\hat G)$
is divisible. Using this property we conclude that, $H\mathbb
Z$-length of $M\rtimes C$ is not greater than $\omega+1$ if and
only if the composite map
$$
\Lambda^2(\hat M)^C\to \Lambda^2(\hat M)_C\to H_2(\eta_\omega)
$$
is an epimorphism (see proposition \ref{keyprop}). Theorem
\ref{Theorem_omega+1} is our main result of section
\ref{hzsection}. There is a simple condition on a tame $\mathbb
Z[C]$-module $M$ which implies that $H\mathbb Z\text{-}{\sf
length}(M\rtimes C)\leq \omega+1$. Theorem \ref{Theorem_omega+1}
provides a large class of groups for which one can describe
$H\mathbb Z$-localization explicitly. In particular, we show that,
if the homology $H_2(M\rtimes C)$ is finite, then the module $M$
satisfies the condition (ii) of Theorem \ref{Theorem_omega+1} and
therefore, for such $M$, $H\mathbb Z\text{-}{\sf length}(M\rtimes
C)\leq \omega+1.$

In section \ref{boussection} we recall the method of Bousfield
from \cite{Bousfield}, which gives a possibility to study the
second homology of the pro-nilpotent completion of a free group.
In section \ref{mainsection} we present our root examples
(\ref{rootgroup}) and (\ref{rootgroup1}) and prove that they have
$H\mathbb Z$-length greater than $\omega+1$. Following the scheme
described above, we get the same result for a free non-cyclic
group. In the last section of the paper we present an alternative
approach for proving that some groups have a long $H\mathbb
Z$-localization tower. Consider the wreath product
$$\mathbb Z\wr\mathbb Z=\langle a,b\ |\ [a,a^{b^i}]=1,\ i\in
\mathbb Z\rangle$$ Using functorial technique, we show in theorem
\ref{zwrz} that $H\mathbb Z\text{-}{\sf length}(\mathbb
Z\wr\mathbb Z)\geq \omega+2$. In the last section, as an
application of the theory developed in the paper, we give an
explicit construction of $EG_{Kl}$.

\section{Relative central extensions and $H\mathbb Z$-localization}

Throughout this section $G,H$ denote groups, $f:H\to G$ a
homomorphism and $A$ an abelian group.

\subsection{(Co)homology of a homomorphism}

Consider the continuous map between classifying spaces $Bf:BH\to
BG$ and its mapping cone ${\sf Cone}(Bf)$. Following Bousfield
\cite[2.14]{Bousfield}, we define homology and cohomology of $f$
with coefficients in $A$ as follows
\begin{align*} & H_n(f,A)=H_n({\sf Cone}(Bf),A),\\ & H^n(f,A)=H^n({\sf Cone}(Bf),A).\end{align*}
Then there are long exact
sequences
$$ \dots \to H_2(H,A) \to H_2(G,A) \to H_2(f,A) \to  H_1(H,A) \to  H_1(G,A) \to H_1(f,A)\to 0,$$
$$0\to H^1(f,A) \to H^1(G,A) \to H^1(H,A)\to H^2(f,A) \to H^2(G,A)\to H^2(H,A)\to \dots.$$
In particular, $H_1(f)={\rm Coker}\{H_{\sf ab}\to G_{\sf ab}\}$
and $H_1(f,A)=H_1(f)\otimes A.$

We denote by $\bar C^\bullet(G,A)$ the  complex of normalized
cochains of $G$ with coefficients in $A,$ \cite[6.5.5]{Weibel} by
$\partial^n:\bar C^n(G,A)\to \bar C^{n+1}(G,A)$ its differential
and by $\bar Z^n(G,A)$ and $\bar B^n(G,A)$ the groups of
normalized cocycles and coboundaries. For a homomorphism $f:H\to
G$ and an abelian group $A$ we denote by $\bar Z^n(f,A)$ and $\bar
B^n(f,A)$ the following subgroups of $\bar Z^n(G,A)\oplus  \bar
C^{n-1}(H,A)$ \begin{align*} & \bar Z^n(f,A)=\{ (c,\alpha) \mid
f^*c= -\partial\alpha \},\\ & \bar B^n(f,A)=\{ (-\partial\beta,
f^*\beta+\partial\gamma ) \mid \beta\in \bar C^{n-1}(G,A),
\gamma\in \bar C^{n-2}(H,A) \}.
\end{align*}
Since the map $\partial:\bar C^0(H,A)\to \bar C^1(H,A)$ is trivial, we have
$$\bar B^2(f,A)=\{ (-\partial\beta, \beta f) \mid \beta\in \bar C^{1}(G,A) \}.$$
\begin{Lemma} For $n\geq 1$ there is an isomorphism $H^n(f,A)\cong \bar Z^n(f,A)/\bar B^n(f,A).$
\end{Lemma}
\begin{proof} For a space $X$, we denote by $C_\bullet(X)$ the complex of integral chains.
Then $C_\bullet(X,A)=C_\bullet(X)\otimes A$ and $C^\bullet(X,A)={\sf Hom}(C_\bullet(X),A).$
For a continuous map $F:X\to Y$ we denote by $C_\bullet(F):C_\bullet(X)\to C_\bullet(Y)$ the induced morphism of complexes.
Then there is a natural homotopy equivalence of complexes ${\sf Cone}(C_\bullet(F))\simeq C_\bullet({\sf Cone}(F))$.
It follows that there is a natural homotopy equivalence of complexes
\begin{equation}\label{cone_cont}
C^\bullet({\sf Cone}(F),A)  \simeq  {\sf Cone}( C^\bullet(F,A))[-1].
\end{equation}

Denote by $C^\bullet(G,A)$ the complex of (non-normalised) cochains of the group $G$. For a homomorphism $f:H\to G$ we denote by $C^\bullet(f,A):C^\bullet(G,A)\to C^\bullet(H,A)$ the induced morphism of complexes. There is a natural homotopy equivalence  $C^\bullet(G,A)\simeq C^\bullet(BG,A)$. Moreover, there is a natural homotopy equivalence of complexes of normalised and non-normalised cochains $\bar C^\bullet(G,A)\simeq C^\bullet(G,A)$ because that come from two different functorial resolutions. It follows that there is a natural homotopy equivalence ${\sf Cone}(\bar C^\bullet(f,A))\simeq {\sf Cone}( C^\bullet(Bf,A)).$ Combining this with \eqref{cone_cont}
we get
$
C^\bullet({\sf Cone}(Bf),A)\xrightarrow{\sim}  {\sf Cone}( \bar C^\bullet(f,A))[-1].
$
The assertion follows.
\end{proof}

\subsection{Relative central extensions}

\begin{Definition}
A relative central extension of $G$ by $A$ with respect to $f$ is a couple $\mathcal E=($\hbox{$A\overset{\iota}\mono E\overset{\pi}\epi G$}$,\tilde f),$
where $A\overset{\iota}\mono E\overset{\pi}\epi G$ is a central extension of $G$ and  $\tilde f:H\to E$ is a homomorphism such that $\pi \tilde f=f.$
$$
\xyma{ &  & H \ar@{->}[d]^{\tilde f} \ar@{->}[rd]^f\\
0 \ar@{->}[r] & A \ar@{->}[r]^\iota & E\ar@{->}[r]^\pi &
G\ar@{->}[r] & 1 }
$$
Two relative central extensions $(A\overset{\iota_1}\mono
E_1\overset{\pi_1}\epi G, \tilde f_1)$ and
$(A\overset{\iota_2}\mono E_2\overset{\pi_2}\epi G,\tilde f_2)$
are said to be equivalent if there exist an isomorphism
$\theta:E_1\overset{\cong}\to E_2$ such that $\theta
\iota_1=\iota_2,$ $\pi_2\theta=\pi_1$ and $\theta\tilde f_1=\tilde
f_2.$
 \end{Definition}

Let $(c,\alpha)\in \bar Z^2(f,A).$ Consider the central extension
$A\mono E_{c}\epi G$ corresponding to the $2$-cocycle $c.$ The
underling set of $E_c$ is equal to $A\times G$ and the product is
given by $$(a_1,g_1)(a_2,g_2)=(a_1+a_2+c(g_1,g_2),g_1g_2).$$
Denote by $\tilde f_\alpha:H\to E_c$ the map given by $\tilde
f_\alpha(h)=(\alpha(h),f(h)).$ Note that the equality
$f^*c=-\partial \alpha$ implies the equality
$$c(f(h_1),f(h_2))=-\alpha(h_1)+\alpha(h_1h_2)-\alpha(h_2)$$ for all
$h_1,h_2\in H.$ It follows that $\tilde{f}_\alpha$ is a
homomorphism. Indeed \begin{multline*}\tilde f_\alpha(h_1)\tilde
f_\alpha(h_2)=(\alpha(h_1),f(h_1))(\alpha(h_2),f(h_2)) =\\
(\alpha(h_1)+\alpha(h_2)+c(f(h_1),f(h_2)),f(h_1)f(h_2)) =\\
(\alpha(h_1h_2),f(h_1h_2))
=\tilde{f}_\alpha(h_1h_2).\end{multline*} Then we obtain a
relative central extension
$$\mathcal E(c,\alpha)=(A\mono E_c \epi G,\tilde f_\alpha).$$

\begin{Proposition}\label{proposition_bijection_H_2} The map $(c,\alpha)\mapsto \mathcal E(c,\alpha)$ induces a bijection between elements of $H^2(f,A)$ and equivalence classes of relative central extensions of $G$ by $A$ with respect to $f$.
\end{Proposition}
\begin{proof}
Any central extension is equivalent to the extension $A\mono
E_{c}\epi G$ for a normalised  2-cocycle $c.$ Hence, it is
sufficient to consider only them. Consider a relative central
extension $(A\mono E_{c}\epi G,\tilde f).$ Define $\alpha:H \to A$
so that $\tilde f(h)=(\alpha(h),f(h)).$ Since
$\tilde{f}(1)=(0,1),$ $\alpha$ is a normalised 1-cochain. Since
$\tilde f$ is a homomorphism, we get
$$(\alpha(h_1)+\alpha(h_2)+c(f(h_1),f(h_2)),f(h_1)f(h_2))=(\alpha(h_1h_2),
f(h_1h_2)).$$
 Thus $f^*c=-\partial \alpha.$ It follows that any relative central extension is isomorphic to the relative central extension $\mathcal E(c,\alpha)$ for some $(c,\alpha)\in Z^2(f,A).$

Consider two elements $(c,\alpha),(c',\alpha')\in Z^2(f,A)$ such
that $(c',\alpha')-(c,\alpha)=(-\partial \beta,\beta f)$ for some
$\beta\in \bar C^1(G,A).$ It follows that
$$c'(g_1,g_2)+\beta(g_1)+\beta(g_2)=c(g_1,g_2)+\beta(g_1g_2)$$ for
any $g_1,g_2\in G$ and $\alpha'(h)=\alpha(h)+\beta(f(h))$ for any
$h\in H.$ Denote by $\theta_\beta:E_c\to E_{c'}$ the map given by
$\theta_\beta(a,g)=(a+\beta(g),g).$ Then $\theta_\beta$ is a
homomorphism. Indeed,
\begin{multline*}\theta_\beta(a_1,g_1)\theta_\beta(a_2,g_2)=(a_1+\beta(g_1),g_1)(a_2+\beta(g_2),g_2)=\\
(a_1+a_2+\beta(g_1)+\beta(g_2)+c'(g_1,g_2),g_1g_2)=\\
(a_1+a_2+\beta(g_1g_2)+c(g_1,g_2),g_1g_2)=\theta_{\beta}((a_1,g_1)(a_2,g_2)).\end{multline*}
Moreover, $\theta_\beta$ is an isomorphism, because
$\theta_{-\beta}$ is its inverse, and it is easy to see that it is
an equivalence of the relative central extensions. It follows that
the map $(c,\alpha)\mapsto \mathcal E(c,\alpha)$ induces a
surjective map from elements of $H^2(f,A)$ to equivalence classes
of relative central extensions.

Consider two elements $(c,\alpha),(c',\alpha')\in Z^2(f,A)$ such that the relative central extensions $\mathcal E(c,\alpha)$ and
 $\mathcal E(c',\alpha')$ are equivalent. Then there is an equivalence $\theta:E_c\to E_{c'}.$ Since $\theta$ respects the i
 njections from $A,$ $\varphi(a,1)=(a,1)$ for any $a\in A.$ Since $\theta$ respects  the projections on $G,$ there exist a unique
 normalised 1-cochain $\beta:G\to A$ such that $\theta(0,g)=(\beta(g),g)$ for any $g\in G.$ Using that $c$ is a normalised 2-cocycle, we obtain
 $$\theta(a,g)=\theta((a,1)(0,g))=(a,1)(\beta
(g),g)=(a+\beta(g),g).$$ Then the fact that $\theta$ is a
homomorphism implies that
$$c'(g_1,g_2)+\beta(g_1)+\beta(g_2)=c(g_1,g_2)+\beta(g_1g_2),$$ and
hence $c'-c=-\partial \beta,$ and the equality
 $\theta\tilde f_\alpha=\tilde{f}_{\alpha'}$ implies
 $\alpha'-\alpha=\beta f.$ The assertion follows.
\end{proof}

\subsection{Universal relative central extensions}

Let $A_1$ and $A_2$ be abelian groups. Recall that c morphism from a central extension
 $A_1\overset{\iota_1}\mono E_1\overset{\pi_1}\epi G$
 to a central extension
 $A_2\overset{\iota_2}\mono E_2 \overset{\pi_2}\epi G$ is a couple
 $(\varphi,\theta),$ where  $\varphi:A_1\to A_2$ and
 $\theta:E_1\to E_2$
are homomorphisms such that
 $\theta \iota_1=\iota_2\tau,$ $ \pi_2 \theta=\pi_1.$
\begin{Lemma}[{cf. \cite[Lemma 6.9.6]{Weibel}}]\label{lemma_restriction_to_commutator} Let $(\varphi,\theta)$ and $(\varphi',\theta')$ be morphisms
from a central extension $A_1\mono E_1\epi G$
 to a central extension
 $A_2\mono E_2 \epi G.$ Then the restrictions on the commutator subgroup coincide  $\theta|_{[E_1,E_1]}=\theta'|_{[E_1,E_1]}.$
\end{Lemma}
\begin{proof}
For the sake of simplicity we identify $A_1$ with the subgroup of
$E_1$ and $A_2$ with the subgroup of $E_2.$ Consider the map
$\rho:E_1\to A_2$ given by $\rho(x)=\theta(x)\theta'(x)^{-1}.$
Since $A_2$ is central, we get
\begin{multline*}\rho(x)\rho(y)=\theta(x)\theta'(x)^{-1}\rho(y)=\theta(x)\rho(y)\theta'(x)^{-1}=\\ \theta(x)\theta(y)\theta'(y)^{-1}\theta'(x)^{-1}=
\theta(xy)\theta'(xy)^{-1}=\rho(xy).\end{multline*} Hence $\rho$
is a homomorphism to an abelian group. Thus $\rho|_{[E_1,E_1]}=1.$
\end{proof}

\begin{Lemma}\label{lemma_splits} Let $(\varphi,\theta)$  be a morphism from a central extension $A_1\overset{\iota_1}\mono E_1 \overset{\pi_1}\epi G$
 to a central extension
 $A_2\overset{\iota_2}\mono E_2 \overset{\pi_2}\epi G.$ If $\varphi=0,$ then $A_2\mono E_2 \epi G$ splits.
\end{Lemma}
\begin{proof}
$\varphi=0$ implies $\theta \iota_1=0.$ Since $G$ is a cokernel of $\iota_1,$ there exists  $s:G\to E_2$ such that $s\pi_1=\theta.$
Then $\pi_2 s\pi_1=\pi_2\theta=\pi_1.$ It follows that $\pi_2 s={\sf id}.$
\end{proof}
\begin{Definition}
A morphism from a relative central extension
 $(A_1\mono E_1\epi G,\tilde f_1)$
 to a relative central extension
 $(A_2\mono E_2 \epi G,\tilde f_2)$ is a morphism
 $(\varphi,\theta)$  from the central extension $A_1\mono E_1\epi G$ to the central extension $A_2\mono E_2\epi G$ such that $\theta \tilde f_1=\tilde f_2.$ So relative central extensions of $G$ with respect to $f$ form a category. The initial object of this category is called the {\it universal relative central extension} of $G$ with respect to $f.$
\end{Definition}

\begin{Example}\label{example_E_phi} For any homomorphism $\varphi:A_1\to A_2$ and any $(c,\alpha)\in Z^2(f,A_1)$ there is a morphism
of relative central extensions
\begin{equation}
\mathcal E(\varphi):\mathcal E(c,\alpha)\longrightarrow \mathcal E(\varphi c,\varphi\alpha), \hspace{1cm} \mathcal E(\varphi)=(\varphi,\varphi\times {\sf id}).
\end{equation}
\end{Example}

\begin{Definition}
A homomorphism $f:H\to G$ is said to be {\it perfect} if $f_{\sf ab}:H_{\sf ab}\to G_{\sf ab}$ is an epimorphism. In other words, $f$ is perfect if and only if
$H_1(f)=0.$
\end{Definition}

\begin{Lemma}[{cf. \cite[Lemma 6.9.6]{Weibel}}]\label{lemma_perf_unique} Let $(\varphi,\theta)$ and $(\varphi',\theta')$ be morphisms from a relative central extension $(A_1\mono E_1\epi G,\tilde f_1)$
 to a relative central extension
 $(A_2\mono E_2 \epi G,\tilde f_2).$
  If $\tilde f_1$ is perfect, then $(\varphi,\theta)=(\varphi',\theta').$
\end{Lemma}
\begin{proof} Since $\tilde f_1$ is perfect, we obtain ${\rm Im}(\tilde f_1)[E_1,E_1]=E_1.$
Since $\theta \tilde f_1=\tilde f_2=\theta' \tilde f_1,$ we have $\theta|_{{\rm Im}(\tilde f_1)}=\theta'|_{{\rm Im}(\tilde f_1)}.$ Lemma
\ref{lemma_restriction_to_commutator} implies $\theta|_{[E_1,E_1]}=\theta'|_{[E_1,E_1]}.$ The assertion follows.
\end{proof}

\begin{Proposition}\label{proposition_univ_central} The universal relative central extension of $G$ with respect to $f$ exists if and only if $f$
is perfect. Moreover, in this case it is unique (up to isomorphism) and given by a short exact sequence
$$
\xyma{& & H\ar@{->}[d]^u \ar@{->}[rd]^f\\ 0\ar@{->}[r] & H_2(f)
\ar@{->}[r] & U \ar@{->}[r] & G \ar@{->}[r] & 1}
$$
where $u:H\to U$ is a perfect homomorphism.
\end{Proposition}

\begin{proof}[Proof of Proposition \ref{proposition_univ_central}]
Assume that there is a universal relative central extension $(A\mono U \overset{\pi}\epi G, u)$ of $G$ with respect to $f$ and prove that $f$ is perfect.
Set $B={\rm Coker}(f_{\sf ab}:H_{\sf ab}\to G_{\sf ab})$ and consider the epimorphism $\tau:U\epi B$ given by the composition $U\epi G\epi G_{\sf ab} \epi B. $
Then we have two morphisms $(0,\binom{\tau}{\pi})$ and $(0,\binom{0}{\pi})$ from $(A\mono U \epi G, u)$ to the split extension $(B \mono B\times G \epi G,
\binom{0}{f}).$ The universal property implies that they are equal, and hence $B=0.$ Thus $f$ is perfect.

Assume that $f$ is perfect. Since $H_1({\sf Cone}(Bf))=H_1(f)=0,$ the universal coefficient formula for the space ${\sf Cone}(Bf)$ implies that there is an
isomorphism $H^2(f,A)\cong {\sf Hom}(H_2(f),A)$ natural by $A.$
Chose an element $(c_u,\alpha_u)\in Z^2(f,H_2(f))$ that represents the element of $H^2(f,H_2(f))$ corresponding to the identity map in ${\sf Hom}(H_2(f),H_2(f)).$
Set $U:=E_{c_u},$  $u=\tilde f_{\alpha_u}$ and $\mathcal E_u=\mathcal E(c_u,\alpha_u).$
Then $\mathcal E_u=(H_2(f)\overset{\iota_u}\mono U\overset{\pi_u}\epi G,u).$ Take a homomorphism $\varphi:H_2(f)\to A$ and consider the commutative diagram
$$
\xyma{ Z^2(f,H_2(f)) \ar@{->}[d]^{\varphi_*}\ar@{->>}[r] &
H^2(f,H_2(f)) \ar@{->}[d]^{\varphi_*} & {\sf Hom}(H_2(f),H_2(f))
\ar@{->}[d]^{\varphi \circ -}\ar@{->}[l]^\cong\\
Z^2(f,A) \ar@{->>}[r] & H^2(f,A) & {\sf Hom}(H_2(f),A).
\ar@{->}[l]^\cong }
$$
It shows that the isomorphism ${\sf Hom}(H_2(f),A)\cong H^2(f,A)$ sends $\varphi$ to the class of $(\varphi c_u,\varphi \alpha_u).$ Combining this with Proposition \ref{proposition_bijection_H_2} we obtain that any relative central extension is isomorphic to the extension $\mathcal E(\varphi c_u,\varphi \alpha_u)$ for some $\varphi:H_2(f)\to A.$ For any relative central extension $\mathcal E(\varphi c_u,\varphi \alpha_u)$ there exists a morphism $\mathcal E(\varphi):\mathcal E_u\to \mathcal E(\varphi c_u,\varphi \alpha_u)$ from
 Example \ref{example_E_phi}. Then we found a morphism from $\mathcal E_u$ to any other relative central extension. In order to prove that $\mathcal E_u$ is the universal relative central extension, we have to prove that such a morphism  is unique. By Lemma \ref{lemma_perf_unique} it is enough to prove that $u:H\to U$ is perfect.

Prove that $u:H\to U$ is perfect. In other words we  prove that \hbox{${\rm Im}(u)[U,U]=U.$} Set $E:={\rm Im}(u)[U,U],$ $A=\iota_u^{-1}(E) $ and $\tilde f:H\to E$ given by $\tilde f(h)=u(h).$ Note that $\tilde f$ is perfect. Since $f$ is perfect, $\pi_u(E)= {\rm Im}(f)[G,G]=G.$ Consider the restriction $\pi=\pi_u|_{E}$ and the relative central extension
${\mathcal E}=(A\mono E \overset{\pi }\epi G, \tilde f)$ with the obvious embedding ${\mathcal E} \hookrightarrow \mathcal E_u.$  Consider the projection $\varphi': H_2(f)\to H_2(f)/A$ and take the
 composition $$\mathcal E \hookrightarrow \mathcal E_u \xrightarrow{\mathcal E(\varphi')} \mathcal E(\varphi'c_u,\varphi'\alpha_u).$$ The composition is equal to $(0, (\varphi'\times 1)|_E).$  By Lemma \ref{lemma_splits} $E(\varphi'c_u,\varphi'\alpha_u)$ splits. Thus $(\varphi'c_u,\varphi'\alpha_u)$ represents $0$ in $H^2(f,A)\cong {\sf Hom}(H_2(f),A),$ and hence
$\varphi'=0.$ It follows that $A=H_2(f).$ Then the extension
$A\mono E \epi G$ is embedded into the extension $A \mono U \epi
G.$ It follows that $E=U.$
\end{proof}
\begin{Remark}\label{remark_H2}
If $f_{\sf ab}:H_{\sf ab}\to G_{\sf ab}$ is an isomorphism, then
$H_2(f)={\rm Coker}\{H_2(H)\to H_2(G)\}.$
\end{Remark}
\begin{Remark}\label{remark_univ_central} In the proof of Proposition \ref{proposition_univ_central} we show that, if $f$ is perfect, then the universal relative central extension corresponds to the identity map ${\sf Hom}(H_2(f),H_2(f))$ with respect to the isomorphism $H^2(f,H_2(f))\cong{\sf Hom}(H_2(f),H_2(f))$ that comes from the universal coefficient theorem.
\end{Remark}

\subsection{HZ-localization tower via relative central extensions}

Here we give an approach to the H$\mathbb Z$-localization tower
\cite{Bousfield} via relative central extensions.

Let $G$ be a group and $\eta:G\to EG$ be its H$\mathbb
Z$-localization. For an ordinal $\alpha$ we define $\alpha$th term
of the H$\mathbb Z$-localization tower by $T_\alpha
G:=EG/\gamma_\alpha(EG),$ where $\gamma_\alpha(EG)$ is the
$\alpha$th term of the transfinite lower central series (see
\cite[Theorem 3.11]{Bousfield}). By $\eta_\alpha:G\to T_\alpha G$
we denote the composition of $\eta$ and the canonical projection,
and by $t_\alpha:T_{\alpha+1}G\epi T_{\alpha}G$ we denote the
canonical projection. The main point of \cite{Bousfield} is that
$T_\alpha G$ can be constructed inductively and $EG=T_\alpha G$
for big enough $\alpha$. We threat the construction of $T_\alpha
G$ for a non-limit ordinal $\alpha$ via universal relative central
extensions.

\begin{Proposition} Let $G$ be a group and $\alpha>1$ be an ordinal. Then $(\eta_\alpha)_{\sf ab}: G_{\sf ab}\to (T_\alpha G)_{\sf ab}$ is an isomorphism, the universal central extension of $T_\alpha G$ with respect to $\eta_\alpha$ is given by
$$\xyma{ & & G\ar@{->}[d]^{\eta_{\alpha+1}}\ar@{->}[rd]^{\eta_\alpha} & & & \\
0\ar@{->}[r] & H_2(\eta_\alpha) \ar@{->}[r] & T_{\alpha+1} G
\ar@{->}[r]^{t_\alpha} & T_{\alpha} G \ar@{->}[r] & 1, }
$$
and $H_2(\eta_\alpha)={\rm Coker}\{H_2(G)\to H_2(T_\alpha G)\}.$
\end{Proposition}
\begin{proof}
It follows from  \cite[3.2]{Bousfield}, \cite[3.4]{Bousfield}, Proposition  \ref{proposition_univ_central} and Remarks \ref{remark_H2}, \ref{remark_univ_central}.
\end{proof}

\section{Homology of stem-extensions}
Consider a central extension of groups
\begin{equation}\label{stem}
1\to N\to G\to Q\to 1
\end{equation}
It is shown in \cite{Eckmann1} that, there is a natural long exact
sequence
\begin{equation}\label{delta1}
H_3(G)\to H_3(Q)\rightarrow (G_{ab}\otimes N)/U\to H_2(G)\to
H_2(Q)\buildrel{\beta}\over\rightarrow N\to H_1(G)\to H_1(Q)\to 0
\end{equation}
where $U$ is the image of the natural map
$$
H_4K(N,2)\to G_{ab}\otimes N.
$$
Here $H_4K(N,2)$ is the forth homology of the Eilenberg-MacLane
space $K(N,2)$ which can be described as the Whitehead quadratic
functor $$ H_4K(N,2)=\Gamma^2N.
$$

A central extension (\ref{stem}) is called a {\it stem-extension}
if $N\subseteq [G,G]$. For a stem extension (\ref{stem}), the
exact sequence (\ref{delta1}) has the form (see \cite{Eckmann},
\cite{Eckmann1})
\begin{equation}\label{delta}
H_3(G)\to H_3(Q)\buildrel{\delta}\over\rightarrow G_{ab}\otimes
N\to H_2(G)\to H_2(Q)\buildrel{\beta}\over\rightarrow N\to 0
\end{equation}
The map $\delta$ is given as follows. We present (\ref{stem}) in
the form
$$
1\to S/R\to F/R\to F/S\to 1,
$$
for a free group $F$ and normal subgroups $R,S$ with $R\subset S,$
$[F,S]\subseteq R.$ Then the map $\delta$ is induced by the
natural epimorphism $S_{ab}\to S/R$:
$$
H_3(Q)=H_1(F/S, S_{ab})\to H_1(F/S, S/R)=Q_{ab}\otimes
N=G_{ab}\otimes N.
$$
The isomorphism $H_1(F/S, S/R)=Q_{ab}\otimes N$ follows from the
triviality of $F/S$-action on $S/R$.

 In this section we consider the class of metabelian
groups of the form $Q=M\rtimes C$, where $C$ is an infinite cyclic
group and $M$ a $\mathbb Z[C]$-module. It follows immediately from
the homology spectral sequence that, for any $n\geq 2$, there is a
short exact sequence
$$
0\to H_0(C, H_n(M))\to H_n(Q)\to H_1(C, H_{n-1}(M))\to 0
$$
which can be presented in terms of (co)invariants as
$$
0\to H_n(M)_C\to H_n(Q)\to H_n(M)^C\to 0.
$$
Composing the last epimorphism with $H_{n-1}(M)^C\to
H_{n-1}(M)_C\hookrightarrow H_{n-1}(Q)$, we get a natural (in the
category of $\mathbb Z[C]$-modules) map
$$
\alpha_n: H_n(Q)\to H_{n-1}(Q).
$$
In the next proposition we will construct a composite map
\begin{equation}\label{alpha3}
\alpha_3': H_3(Q)\to H_1(C, \Lambda^2(M))\hookrightarrow
\Lambda^2(M)\twoheadrightarrow \Lambda^2(M)_C\hookrightarrow
H_2(Q)
\end{equation}
using group-theoretical tools, without spectral sequence.
Probably, $\alpha_3'$ coincides with $\alpha_3$ up to isomorphism,
but we will not use this comparison later.
\begin{Proposition}\label{alpha33} For a stem extension (\ref{stem}) of a group
$Q=M\rtimes C$, there exists a map $\alpha_3': H_3(Q)\to H_2(Q)$,
given as a composition (\ref{alpha3}), such that the following
diagram is commutative
$$
\xyma{H_3(Q)\ar@{->}[d]^\delta \ar@{->}[r]^{\alpha_3'} & H_2(Q)\ar@{->}[d]^\beta\\
G_{ab}\otimes N\ar@{->}[r] & N}
$$
where the lower horizontal map is the projection
$$
G_{ab}\otimes N\to C\otimes N=N.
$$
\end{Proposition}
\begin{proof}
We choose a free group $F$ with normal subgroups $R\subset
S\subset T$ such that
$$
F/T=C,\ F/S=Q, F/R=G.
$$
In the above notation, we get $M=T/S, N=S/R$. The proof follows
from the direct analysis of the following diagram, which corners
are exactly the roots of the diagram given in proposition:
\begin{equation}\label{dia22}
\xyma{H_3(Q)\ar@/^80pt/[rrrd]_{\alpha_3'}_{\;}="d"\ar@{=}[d] & H_1(F/T, \frac{S\cap T'}{S'})\ar@{->}[r]\ar@{=}[d] & H_1(F/T, \frac{S\cap T'}{[S,T]})\ar@{>->}[d]\\
H_1(F/S, S_{ab})\ar@{->}[d]^\delta \ar@{->}[r] & H_1(F/T, S_{ab})
& H_2(T/S)\ar@{->}[r] & H_2(Q)\ar@{->}[d]^{\beta}\\ H_1(F/S,
S/R)\ar@{=}[r] & (F/S)_{ab}\otimes S/R\ar@{->}[rr] & & S/R}
\end{equation}
All arrows of this diagram are natural. We will make comments only
about two maps from the diagram, other maps are obviously defined.
The map
$$
H_1(F/T, \frac{S\cap T'}{S'})\to H_1(F/T, S_{ab})
$$
is an isomorphism since
$$
H_1(F/T, \frac{S}{S\cap T'})\hookrightarrow H_1(F/T,
T_{ab})=H_3(F/T)=0.
$$
The vertical map in the diagram $H_1(F/T, \frac{S\cap
T'}{[S,T]})=H_1(F/T, H_2(S/T))\hookrightarrow H_2(S/T)$ follows
from the identification of $H_1$ of a cyclic group with
invariants.

The commutativity of (\ref{dia22}) follows from the commutativity
of the natural square
$$
\xyma{H_1(F/S, S_{ab}) \ar@{->>}[r]\ar@{->}[d] & H_1(F/T, S_{ab})\ar@{->}[d]\\
H_1(F/S, S/R)\ar@{->>}[r] & H_1(F/T, S/R)}
$$
and identification of $H_1(F/T, -)$ with invariants of the
$F/T$-action.
\end{proof}

\section{Tame modules and completions}\label{tamemodules}

Throughout the section $C$ denotes an infinite cyclic group. If
$R$ is a commutative ring, $R[C]$ denotes the group algebra over
$R$. We use only $R=\mathbb Z,\mathbb Q,\mathbb C.$ The
augmentation ideal is denoted by $I.$ If $t$ is one of two
generators of $C$, $R[C]=R[t,t^{-1}]$ and $I=(t-1).$

\subsection{Finite dimensional $K[C]$-modules}

Let $K$ be a field (we use only $K=\mathbb Q,\mathbb C$), $V$
be a right $K[C]$-module such that ${\sf dim}_K\:V<\infty$. If we
fix a generator $t\in C$ we obtain a linear map $\cdot t : V\to V$
that defines the module structure. We denote the linear map by
$a_V\in {\sf GL}(V).$ The characteristic and minimal polynomials
of $a_V$ are denoted by $\chi_V$ and $\mu_V$ respectively. These
polynomials depend of the choice of $t\in C.$ Note that for any
such modules $V$ and $U$ we have

\begin{equation}\label{eq_chi_mu_sum}
\chi_{V\oplus U}=\chi_V \chi_U, \hspace{1cm} \mu_{V\oplus U}={\sf
lcm}(\mu_V,\mu_U).
\end{equation}

\begin{Lemma}\label{Lemma_structure_theorem_fields} Let $V$ be a right $K[C]$-module such that ${\sf dim}_K\:V<\infty$ and $t\in C$ be a generator.
Then there exist distinct irreducible monic polynomials $f_1,\dots,f_l\in K[x]$ and an isomorphism
$$V\cong V_1\oplus \dots \oplus V_l,$$
where
$$V_i= K[C]/(f_i^{m_{i,1}}(t)) \oplus \dots \oplus K[C]/(f_i^{m_{i,l_i}}(t)),$$
and $m_{i,1}\geq m_{i,2} \geq \dots \geq m_{i,l_i}\geq 1.$
 Moreover, if we set $m_i=\sum_j m_{i,j},$ then $\chi_V=f_1^{m_1}\dots f_l^{m_l}$ and $\mu_V=f_1^{m_{1,1}}\dots f_l^{m_{l,1}}.$
\end{Lemma}
\begin{proof}
Note that $K[C]=K[t,t^{-1}]$ is the polynomial ring $K[t]$
localised at the element $t.$ Then it is a principal ideal domain.
Then the isomorphism follows from the structure theorem for
finitely generated modules over a principal ideal domain. The
statement about $\chi_V$ and $\mu_V$ follows from the fact that
both  characteristic and minimal polynomials of
$K[C]/(f_i^{m_{i,j}}(t))$ equal to $f_i^{m_{i,j}}$ and the
formulas \eqref{eq_chi_mu_sum}.
\end{proof}

Let $R$ be a commutative ring and $t$ be a generator of $C$. For
an $R[C]$-module $M$ and a polynomial $f\in R[x]$ we set
$$M^f=\{m\in M\mid m\cdot f(t)=0\}.$$
Note that  $M^C=M^{x-1}$. It is easy to see that for any $f,g\in
R[x]$ we have
\begin{equation}\label{eq_Mfg}
(M/M^f)^g=M^{fg}/M^f.
\end{equation}

\begin{Corollary}\label{Corollary_structure_over_field} Let $V$ be a right $K[C]$-module such that ${\sf dim}_K\:V<\infty$ and $t\in C$ be a generator. Assume that $\mu_V=f_1^{m_1} \dots f_l^{m_l},$ where $f_1,\dots,f_l\in K[x]$ are distinct irreducible monic polynomials. Consider the filtration
$$0=F_0V \subset F_1V \subset \dots \subset F_lV =V  $$
given by $F_iV=V^{f_1^{m_1}\dots f_i^{m_i}}.$ Then $$V=\bigoplus_i
F_iV/F_{i-1}V,\ \ \ \mu_{F_iV/F_{j-1}}=f_j^{m_j}\dots f_i^{m_i}$$
and $F_iV/F_{j-1}V=(V/F_{j-1}V)^{f_j^{m_j}\dots f_i^{m_i}}.$
\end{Corollary}
\begin{proof}
In the notation of Lemma \ref{Lemma_structure_theorem_fields} we
obtain $F_iV=V_1\oplus \dots \oplus V_i.$ The assertion follows.
\end{proof}

\subsection{Tame $\mathbb Z[C]$-modules}

The {\it rank} of an abelian group $A$ is ${\sf dim}_{\mathbb
Q}(A\otimes \mathbb Q).$ The torsion subgroup of $A$ is denoted by
${\sf tor}(A).$ The following statement seems to be well known but
we can not find a reference, so we give it with a proof.

\begin{Lemma} Let $A$ be a torsion free abelian group of finite rank and $B$ be a finite abelian group. Then $A\otimes B$ is finite.
\end{Lemma}
\begin{proof}
It is sufficient to prove that $A\otimes \mathbb Z/p^k$ is finite
for any prime $p$ and $k\geq 1.$ Consider a $p$-basic subgroup
$A'$ of $A$ (see \cite[VI]{Fuchs}). Then $A'\cong \mathbb Z^r,$
where $r$ is the rank of $A$ and $A/A'$ is $p$-divisible. Thus
$(A/A') \otimes \mathbb Z/p^k=0.$ Hence the map $(\mathbb
Z/p^k)^r\cong A'\otimes \mathbb Z/p^k \epi A\otimes \mathbb Z/p^k$
is an epimorphism.
\end{proof}

A finitely generated $\mathbb Z[C]$-module $M$ is said to be {\it
tame} if the group $M\rtimes C$ is finitely presented.

\begin{Proposition}[{Theorem C of \cite{Bieri-Strebel_78}}]\label{Proposition_main_tame} Let $M$ be a finitely generated $\mathbb Z[C]$-module. Then $M$ is tame if and only if the following properties hold:
\begin{itemize}
\item ${\sf tor}(M)$ is finite; \item the rank of $M$ is finite;
\item there is a generator $t$ of $C$ such that $\chi_{M\otimes
\mathbb Q}$ is integral.
\end{itemize}
\end{Proposition}

\begin{Lemma}[{Lemma 3.4 of \cite{Bieri-Strebel_78}}]\label{Lemma_tensor_product} Let $M$ and $M'$ be tame $\mathbb Z[C]$-modules. Then $M\otimes M'$ is a
 tame $\mathbb Z[C]$-module (with the diagonal action).
\end{Lemma}

\begin{Definition} Let $M$ be a tame module. The generator $t\in C$ such that   such that $\chi_{M\otimes \mathbb Q}$ is integral is called an {\it integral generator} for $M$. When we consider a tame module, we always denote by $t$ an integral generator for $M.$ We set $a_M:=t\otimes \mathbb Q : M\otimes \mathbb Q \to M\otimes \mathbb Q,$  and denote  by $\chi_M,\mu_M$ the characteristic and the minimal polynomial of $a_M.$ In other words $\chi_M=\chi_{M\otimes \mathbb Q}$ and $\mu_{M}=\mu_{M\otimes \mathbb Q}.$
 \end{Definition}

\begin{Lemma}\label{Lemma_minimal_polynomial} $\mu_M$ is an integral monic polynomial for any tame $\mathbb Z[C]$-module $M.$
\end{Lemma}
\begin{proof}
Let $\chi_M=(x-\lambda_1)^{m_1} \dots (x-\lambda_l)^{m_l}$ for
some distinct $\lambda_1,\dots,\lambda_l\in \mathbb C$ and
$m_i\geq 1.$ Then $\mu_M=(x-\lambda_1)^{k_1} \dots
(x-\lambda_l)^{k_l},$ where $1\leq k_i\leq m_i.$ Since $\chi$ is a
monic integral polynomial, $\lambda_1,\dots,\lambda_l$ are
algebraic integers. It follows that the coefficients of $\mu_M$
are algebraic integers as well. Moreover, they are rational
numbers, because $a_M$ is defined rationally. Using that a
rational number is an algebraic integer iff it is an integer
number, we obtain $\mu_M\in \mathbb Z[x].$
\end{proof}

\begin{Proposition}\label{Proposition_structure_of_tame_modules}
Let $M$ be a torsion free tame $\mathbb Z[C]$-module and
$\mu_M=f_1^{m_1} \dots f_l^{m_l}$ where $f_1,\dots,f_l$ are
distinct irreducible integral monic polynomials and $m_i\geq 1$
for all $i$. Consider the filtration
$$0=F_0M\subset F_1M \subset  \dots \subset F_lM=M$$
given by $F_iM=M^{f_1^{m_1}\dots f_i^{m_i}}.$ Then
$F_iM/F_{j-1}M$ is torsion free and
$\mu_{F_iM/F_{j-1}M}=f_j^{m_j}\dots f_i^{m_i}$ for any $i\geq j.$
Moreover, the corresponding filtration on $M\otimes \mathbb Q$
splits:
$$M\otimes \mathbb Q\cong \bigoplus_{i=1}^l\ (F_iM/F_{i-1}M)\otimes \mathbb Q.$$
\end{Proposition}
\begin{proof} Prove that $F_iM/F_{j-1}M$ is torsion free. Let $ v+F_{j-1}M\in F_iM/F_{j-1}M$ and $nv+F_{j-1}M=0.$ Hence $nv\cdot f_1^{m_1}(t) \dots f_j^{m_j}(t)=0$ in $M.$  Using that $M$ is torsion free we get  $v\cdot f_1^{m_1}(t) \dots f_j^{m_j}(t)=0$, and hence $v+F_{j-1}M=0.$ Thus $F_iM/F_{j-1}M$ is torsion free.
Set $V=M\otimes \mathbb Q,$ $K=\mathbb Q$, apply Corollary
\ref{Corollary_structure_over_field} and note that
$F_iV/F_{j-1}V=(F_iM/F_{j-1}M)\otimes \mathbb Q$. The assertion
follows. Here we use Gauss lemma about integral polynomials: an
irreducible polynomial in $\mathbb Z[x]$ is irreducible in
$\mathbb Q[x].$
\end{proof}

Recall that a module $N$ is said to be nilpotent if $NI^n=0$ for
some $n,$ where $I$ is the augmentation ideal. It is easy to see
that a $\mathbb Z[C]$-module $N$ is nilpotent if and only if
$N^{(x-1)^n}=N$ for some $n.$

\begin{Definition} A $\mathbb Z[C]$-module $M$ is said to be {\it invariant free} if $M^C=0.$
\end{Definition}

\begin{Lemma}\label{Lemma_invariant_free} Let $M$ be a torsion free tame $\mathbb Z[C]$-module. Then the following equivalent.
\begin{enumerate}
\item $M$ is invariant free; \item $M$ does not have non-trivial
nilpotent submodules; \item $M_C$ is finite; \item $a_M-1$ is an
automorphism; \item $\chi_M(1)\ne 0;$ \item $\mu_M(1)\ne 0.$
\end{enumerate}
\end{Lemma}
\begin{proof}(1) $\Leftrightarrow$ (2) and (4) $\Leftrightarrow$ (5) $\Leftrightarrow$ (6) are obvious.
The equality ${\rm Ker}(a_M-1)=M^C \otimes \mathbb Q$ implies (1)
$\Leftrightarrow$ (4). Since $M$ is finitely generated $\mathbb
Z[C]$-module, $M_C$ is a finitely generated abelian group. Then
the equality ${\rm Coker}(a_M-1)=M_C\otimes \mathbb Q$ implies (3)
$\Leftrightarrow$ (4).
\end{proof}

\begin{Corollary}\label{Corollary_nilpotent_submodule} Let $M$ be a torsion free tame $\mathbb Z[C]$-module and $\mu_M=(x-1)^{m} f$, where $f(1)\ne 0$. Then there exists the largest nilpotent submodule  $N\leq M$. Moreover, $\mu_N=(x-1)^m,$ $\mu_{M/N}=f,$ $M/N$ is torsion free and invariant free, and the short exact sequence $N\otimes \mathbb Q \mono M\otimes \mathbb Q \epi (M/N)\otimes \mathbb Q$ splits over $\mathbb Q[C]$.
\end{Corollary}
\begin{proof}
If $\mu_M(1)\ne 0,$ then $M$ is already invariant free, $N=0$ and
there is nothing to prove. If $\mu_M(1)=0,$ then we can decompose
$\mu_M=(x-1)^{m_1} f_2^{m_2}\dots f_l^{m_l}$ into a product of
irreducible polynomials such that $f_i(1)\ne 0$ for $i\geq 2.$
Consider the filtration from Proposition
\ref{Proposition_structure_of_tame_modules}. Then $N=F_1M.$
\end{proof}

\begin{Corollary} Let $M$ be a  tame $\mathbb Z[C]$-module. Then there exists the largest nilpotent submodule  $N\leq M$. Moreover, $M/N$ is invariant free and $(M/N)_C$ is finite.
\end{Corollary}

Recall that a module $N$ is said to be prenilpotent if
$NI^n=NI^{n+1}$ for $n>\!>1.$

\begin{Corollary}\label{Corollary_prenilpotent_submodule} Let $M$ be a tame $\mathbb Z[C]$-module and $\mu_M=(x-1)^{m} f$, where $f(1)\ne 0$. Then there exists a prenilpotent submodule $N\leq M$ such that $M/N$ is torsion free and invariant free. Moreover, ${\sf tor}(N)={\sf tor}(M),$ $\mu_N=(x-1)^m,$  $\mu_{M/N}=f$ and the sequence $N\otimes \mathbb Q \mono M\otimes \mathbb Q \epi (M/N)\otimes \mathbb Q$ splits over $\mathbb Q[C].$
\end{Corollary}

\begin{Lemma}\label{Lemma_polycyclic_modules} Let $M$ be a tame torsion free $\mathbb Z[C]$-module. If $\mu_M(0)\in \{-1,1\},$ then $M$ is finitely generated as an abelian group.
\end{Lemma}
\begin{proof}
It follows from the fact that  $\mathbb Z[t,t^{-1}]/(\mu_M(t))$ is
a finitely generated abelian group.
\end{proof}

\subsection{Completion of tame $\mathbb Z[C]$-modules}

If $M$ is a finitely generated $R[C]$-module, we set $\hat
M=\varprojlim M/MI^i$ and we denote by
$$\varphi=\varphi_M:M \longrightarrow \hat M$$
the natural map to the completion. Note that the functor $M\mapsto
\hat M$ is exact \cite[VIII]{Zariski-Samuel_II} and $\hat M /\hat
MI^i=M/MI^i.$

We set $$\mathbb Z_n=\varprojlim\: \mathbb Z/n^i$$ for any $n\in
\mathbb Z.$ In particular, $\mathbb Z_n=\mathbb Z_{-n},$ $\mathbb
Z_0=\mathbb Z,$ $\mathbb Z_1=0$ and, if $n\geq 2,$ then $\mathbb
Z_n=\bigoplus \mathbb Z_p,$ where $p$ runs over all prime divisors
of $n.$

\begin{Lemma}\label{Lemma_completion} Let $M$ be a tame torsion free invariant free $\mathbb Z[C]$-module and $n=\chi_M(1)$. Then $n^i\cdot M\subseteq MI^{i} $ for any $i\geq 1$ and
 there exists a unique epimorphism of $\mathbb Z[C]$-modules
$\hat  \varphi:M\otimes \mathbb Z_n \epi \hat M$ such that the
diagrams
$$
\xyma{
M \otimes \mathbb Z_n\ar@{->>}[r]^{\hat \varphi} \ar@{->>}[d] & \hat M \ar@{->>}[d] \\
M \otimes \mathbb Z/n^i \ar@{->>}[r] & M/MI^i }
$$
are commutative.
\end{Lemma}
\begin{proof} We identify $M$ with the subgroup of $M\otimes \mathbb Q.$ Corollary \ref{Lemma_invariant_free} implies that $n\ne 0.$ Set $b=a_M-1.$ Then the characteristic polynomial of $b$ is equal to $\chi_b(x)=\chi_M(x+1)$ and if $\chi_b=\sum_{i=0}^d \beta_ix^i,$ then $ \beta_0=n.$ Thus $nx=b(\sum_{i=1}^d \beta_i b^{i-1}(x))$ for any $x\in M.$ It follows that $nM\subseteq b(M).$ Hence $n^iM\subseteq b^i(M)=MI^i$ for any $i\geq 1$ and we obtain homomorphisms $M\otimes  \mathbb Z/n^i \to M/MI^i.$  We define $\hat \varphi$ as the composition  $M\otimes \mathbb Z_n \to \varprojlim (M\otimes \mathbb Z/n^i) \to \hat M.$
Since the rank $M$ is finite, the abelian groups $M\otimes \mathbb
Z/n^i$ are finite.
 Thus we get that the homomorphism $\varprojlim (M\otimes \mathbb Z/n^i) \to \hat M$ is an epimorphism because $\varprojlim^1$ of an inverse sequence of finite groups is trivial.  Then it is sufficient to prove that the homomorphism $M\otimes \mathbb Z_n \to \varprojlim (M\otimes\: \mathbb Z/n^i)$ is an epimorphism. For this it is enough to prove that $M\otimes \mathbb Z_p \to \varprojlim (M\otimes\: \mathbb Z/p^i)$ is an epimorphism for any prime $p.$ Consider a $p$-basic subgroup $B$ of $M$ (see \cite[VI]{Fuchs}). Since $B\cong \mathbb Z^l,$ we get $B\otimes \mathbb Z_p = \varprojlim (B\otimes\: \mathbb Z/p^i).$ Using that $B\otimes \mathbb Z/p^i \epi M\otimes \mathbb Z/p^i$ are epimorphisms of finite groups, we obtain that $\varprojlim (B\otimes \mathbb Z/p^i) \to \varprojlim (M\otimes \mathbb Z/p^i)$ is an epimorphism. Then analysing the diagram
$$
\xyma{
B\otimes \mathbb Z_p\ar@{>->}[r] \ar@{->}[d]^\cong & M \otimes \mathbb Z_p \ar@{->}[d] \\
\varprojlim (B\otimes \mathbb Z/p^i)\ar@{->>}[r] & \varprojlim
(M\otimes \mathbb Z/p^i) }
$$
we obtain that the right vertical arrow is an epimorphism.
\end{proof}

A $\mathbb Z[C]$-module is said to be {\it perfect} if $MI=M.$

\begin{Corollary}\label{Corollary_perfect} Let $M$ be a torsion free tame $\mathbb Z[C]$-module. If $\mu_M(1)\in \{-1,1\},$ then $M$ is perfect.
\end{Corollary}
\begin{proof}
 By Lemma \ref{Lemma_invariant_free}, $M$ is invariant free.  Then Lemma \ref{Lemma_completion} implies $\hat M=0.$ Hence $M$ is perfect.
\end{proof}

\begin{Corollary}\label{perfectm} Let $M$ be a tame $\mathbb Z[C]$-module. If $\mu_M=(x-1)^m f,$ where $f$ is an integral polynomial such that $f(1)\in \{-1,1\},$ then $M$ is prenilpotent.
\end{Corollary}
\begin{proof}
A finite module is always prenilpotent, so we can assume that $M$
has no torsion. Further, by Lemma
\ref{Corollary_nilpotent_submodule}, we can consider the largest
nilpotent submodule $N\leq M$ such that $\mu_N=(x-1)^m$ and
$\mu_{M/N}=f.$ Corollary \ref{Corollary_perfect} implies that
$M/N$ is perfect. Then $N$ and $M/N$ are prenilpotent, and hence,
$M$ is prenilpotent.
\end{proof}

\begin{Proposition}\label{Proposition_omega}
Let $M$ and $M'$  be tame $\mathbb Z[C]$-modules with the same
integral generator $t\in C,$ $\lambda_1,\dots,\lambda_l\in \mathbb
C$ are eigenvalues of $a_M$ and $\lambda'_1,\dots,\lambda'_{l'}\in
\mathbb C$ are eigenvalues of $a_{M'}$.
 Assume that the equality $\lambda_i\lambda'_j=1$ holds only if $\lambda_i=\lambda_j'=1.$ Then the homomorphism
$$ (M\otimes M')_C\longrightarrow ( \hat M \otimes \hat M')_C$$
is an epimorphism.
\end{Proposition}
\begin{proof}
Note that if $M_1\mono M_2 \epi M_3$ is a short exact sequence of
tame modules and $(M_1 \otimes M')_C \to (\hat M_1\otimes \hat
M')_C,$ $(M_3 \otimes M')_C \to (\hat M_3\otimes \hat M')_C$ are
epimorphisms, then $(M_2 \otimes M')_C \to (\hat M_2\otimes \hat
M')_C$ is an epimorphism. Indeed, since the functor of completion
is exact, we have the commutative diagram with exact rows
$$
\xyma{
(M_1 \otimes M')_C\ar@{->}[r] \ar@{->>}[d] & (M_2 \otimes M')_C \ar@{->>}[r] \ar@{->}[d]  & (M_3 \otimes M')_C\ar@{->>}[d]\\
(\hat M_1 \otimes \hat M')_C \ar@{->}[r] & (\hat M_2 \otimes \hat
M')_C \ar@{->>}[r]  & (\hat M_3 \otimes \hat M')_C }
$$
that implies this. Then, using Corollary
\ref{Corollary_prenilpotent_submodule}, we obtain that we can
divide our prove into two parts: (1) prove the statement for the
case of torsion free invariant free modules $M,M'$; (2) prove the
statement for the case of a prenilpotent module $M$ and arbitrary
tame module $M'$. Throughout the proof we use that $(M\otimes
M')_C\cong M\otimes_{\mathbb Z[C]} M'_{\sigma},$ where $M'_\sigma$
is the module with the same underling abelian group $M'$ but with
the twisted action of $C:$ $m*t=mt^{-1}.$

(1) Assume that $M,M'$ are torsion free invariant free tame
$\mathbb Z[C]$-modules.  Lemma  \ref{Lemma_invariant_free} implies
that $\lambda_i\ne 1$ and $\lambda_j'\ne 1$  for all $i,j.$ Then
we have $\lambda_i\lambda_j'\ne 1$ for all $i,j.$ Note that the
eigenvalues of $a_M\otimes a_{M'}$ equal to the products
$\lambda_i\lambda_j,$ and hence $1$ is not an eigenvalue of
$a_M\otimes a_{M'}.$ It follows that ${\sf det}(a_M\otimes
a_{M'}-1)\ne 0.$ Consider the minimal polynomial $\mu$ of the
tensor square $a_M \otimes a_{M'}.$ Since the $a_M\otimes a_{M'}$
is defined over $\mathbb Q,$ the coefficients of $\mu$ are
rational (because they are invariant under the action of the
absolute Galois group). Moreover, $\mu=\prod
(x-\lambda_i\lambda_j')^{k_{i,j}}$ for some $k_{i,j},$ and hence,
its coefficients are algebraic integers. It follows that $\mu$ is
a monic polynomial with integral coefficients. The polynomial
$\mu(x+1)$ is the minimal polynomial for $a_M\otimes a_{M'}-1.$
Let $\mu(x+1)=\sum_{i=0}^k n_{i}x^i.$   Then $n_0={\sf
det}(a_M\otimes a_{M'}-1)\ne 0$ and $n_0 \cdot (M\otimes M')
\subseteq  (M\otimes M')(t-1).$ Since the rank of $M\otimes M'$ is
finite, $(M\otimes M')/n_0(M\otimes M')$ is finite, and hence,
$(M\otimes M')_C=(M\otimes M')/(M\otimes M')(t-1)$ is finite. By
Lemma \ref{Lemma_completion} we have epimorphisms $M\otimes
\mathbb Z_n \epi \hat M$ and $M'\otimes \mathbb Z_{n'} \epi \hat
M',$ where $n={\sf det}(a_M-1)$ and $n'={\sf det}(a_{M'}-1).$ It
is easy to see that $$((M\otimes \mathbb Z_n)\otimes (M'\otimes
\mathbb Z_{n'}))_C=(M \otimes M')_C\otimes (\mathbb Z_n \otimes
\mathbb Z_{n'}).$$ Since $(M\otimes M')_C$ is finite, $(M\otimes
M')_C \to  (M \otimes M')_C\otimes (\mathbb Z_n \otimes \mathbb
Z_{n'})$ is an epimorphism. Then $(M\otimes M')_C \to  (\hat M
\otimes \hat M')_C$ is an epimorphism.

(2) Assume that $M$ is a prenilpotent $\mathbb Z[C]$-module and
$M'$ is a tame $\mathbb Z[C]$-module. Then there exists $i$ such
that $\hat M=M/MI^i.$ Since $(\hat M\otimes \hat M')_C\cong \hat M
\otimes_{\mathbb Z[C]} \hat M'_\sigma,$ we get $$(\hat M\otimes
\hat M')_C\cong ( M/MI^i\otimes \hat M')_C\cong ( M/MI^i\otimes
\hat M'/\hat M'I^i)_C\cong ( M/MI^i\otimes  M'/ M'I^i)_C.$$ It
follows that $(M\otimes M')_C\to (\hat M\otimes \hat M')_C$ is an
epimorphism.
\end{proof}

\begin{Corollary}Let $M$  be a tame $\mathbb Z[C]$-module and $\mu_M=(x-1)^{m}f_1^{m_1}\dots f_l^{m_l} $ for some distinct monic irreducible polynomials $f_1,\dots,f_l\in \mathbb Z[x]$ such that $f_i(1)\ne 0$ and $f_i(0)\notin \{1,-1\}$ for all $1\leq i\leq l.$ Then the the homomorphism
$$ (M^{\otimes 2})_C\longrightarrow ( \hat M^{\otimes 2} )_C$$
is an epimorphism.
\end{Corollary}
\begin{proof} Let $\lambda_1,\dots,\lambda_k$ be roots of $\mu_M$. Assume that $\lambda_i\lambda_j=1.$ Then $\lambda_i$ is an invertible algebraic integer, and hence, the absolute term of its minimal polynomial equals to $\pm 1.$ Thus $\lambda_i$ can not be a root of $f_m$ for $1\leq m\leq l.$ It follows that it is a root of $x-1.$ Then $\lambda_i=\lambda_j=1.$
\end{proof}

\begin{Proposition}\label{onekey}
Let $M,M'$  be tame $\mathbb Z[C]$-modules with the same integral
generator $t\in C,$ $\mu_M=(x-\lambda_1)^{m_1}\dots
(x-\lambda_l)^{m_l}$ for some distinct
$\lambda_1,\dots,\lambda_l\in \mathbb C$ and
$\mu_{M'}=(x-\lambda'_1)^{m'_1}\dots (x-\lambda'_{l'})^{m'_{l'}}$
for some distinct $\lambda'_1,\dots,\lambda'_{l'}\in \mathbb C.$
 Assume that the equality $\lambda_i\lambda'_j=1$ holds only if either $m_i=m_j'=1$ or $\lambda_i=\lambda'_j=1.$ Then the cokernel of the homomorphism
$$(M\otimes M')_C \oplus (\hat M\otimes \hat M')^C\longrightarrow ( \hat M \otimes \hat M')_C$$
is finite.
\end{Proposition}
\begin{proof}  Corollary \ref{Corollary_prenilpotent_submodule} implies that the proof can be divided into proofs of the following two statements: (1)
the statement for torsion free invariant free modules $M,M'$; (2)
if $N\mono M\epi M_0$ is a short exact sequence of tame $\mathbb
Z[C]$-modules such that $N\otimes \mathbb Q\mono M \otimes \mathbb
Q \epi M_0\otimes \mathbb Q$ splits, $N$ is prenilpotent,  and the
statement holds for the couple $M_0,M',$ then it holds for the
couple $M,M'.$

(1) Here we prove that the cokernel of $(\hat M\otimes \hat
M')^C\longrightarrow ( \hat M \otimes \hat M')_C$ is already
finite. Set $n=\chi_M(1)$ and $n'=\chi_{M'}(1).$ Lemma
\ref{Lemma_completion} implies that there are epimorphisms
$M\otimes \mathbb Z_{n}\epi \hat M$ and $M'\otimes \mathbb Z_{n'}
\epi \hat M'.$ Using that $-\otimes (\mathbb Z_{n} \otimes \mathbb
Z_{n'})$ is an exact functor, we obtain that there is an
epimorphism $(M\otimes M')_C \otimes (\mathbb Z_n\otimes \mathbb
Z_{n'}) \epi (\hat M\otimes \hat M')_C.$ Moreover, there is an
epimorphism
$$
{\rm Coker}((M\otimes M')^C\to (M\otimes M')_C)\otimes (\mathbb
Z_n \otimes \mathbb Z_{n'}) \epi {\rm Coker}((\hat M\otimes \hat
M')^C\to ( \hat M\otimes \hat M')_C).
$$
It follows that it is enough to prove that ${\rm Coker}((M\otimes
M')^C\to (M\otimes M')_C)$ is finite. Lemma
\ref{Lemma_tensor_product} implies that $M\otimes M'$ is finitely
generated, and hence, $(M\otimes M')_C$ is a finitely generated
abelian group. It follows that it is enough to prove that
$(M\otimes M')^C\otimes \mathbb C \to (M\otimes M')_C\otimes
\mathbb C$ is an epimorphism. Eigenvalues of $a_M\otimes a_{M'}$
are products $\lambda_i\lambda_j'.$ Assume that
$\lambda_i\lambda_j'=1$ for some $i,j.$ Since $M$ and $M'$ are
invariant free, $\lambda_i\ne 1$ and $\lambda_j'\ne 1.$ Then
$m_i=1=m_j'.$ It follows that all Jordan blocks of $a_M\otimes
\mathbb C$ corresponding to $\lambda_i$ and all Jordan blocks of
$a_{M'}\otimes \mathbb C$ corresponding to $\lambda_j'$ are
$1\times 1$-matrices. It follows that all Jordan blocks of
$a_M\otimes a_{M'}\otimes \mathbb C$ corresponding to $1$ are
$1\times 1$-matrices. Hence all Jordan blocks of $B:=a_M\otimes
a_{M'}\otimes \mathbb C -1$ corresponding to $0$ are $1\times
1$-matrices. It is easy to see that, if all Jordan blocks of a
complex linear map $B:V\to V$ corresponding to $0$ are $1\times
1$-matrices, then  $V={\rm Ker}(B) \oplus {\rm Im}(B).$  It
follows that the map ${\rm Ker}(B)\to  {\rm Coker}(B)$ is an
isomorphism. Then $(M\otimes M')^C\otimes \mathbb C \to (M\otimes
M')_C\otimes \mathbb C$ is an isomorphism.

(2) Note that $\hat N=N/NI^i$ for some $i>\!>1.$ Since, $(\hat
N\otimes \hat M')_C$ can be interpret as  $\hat N\otimes_{\mathbb
Z[C]} \hat M'_\sigma,$ (tensor product over $\mathbb Z[C]$),
 we obtain  $(\hat N\otimes \hat M')_C=(N/NI^i\otimes M'/M'I^i)_C.$
 It follows that $(\hat N\otimes \hat M')_C$ is a finitely generated abelian group and the map $(N\otimes M')_C\to (\hat N\otimes \hat M')_C$ is an epimorphism.
 Set \begin{align*} & \mathcal N:=N \otimes M',\ \ \mathcal M:=M\otimes M',\\ & \mathcal M_0:=M_0\otimes
 M',\ \ \tilde{\mathcal N}:=\hat N \otimes \hat M',\\
&  \tilde{\mathcal M}:=\hat M\otimes \hat M',\ \ \
\tilde{\mathcal M}_0:=\hat M_0\otimes \hat
 M',\\
 & \tilde{\mathcal L}:= {\rm Ker}(\tilde{\mathcal M}\epi \tilde{\mathcal M}_0).
 \end{align*}
 Then $\tilde{\mathcal N}_C$ is a finitely generated abelian group and the map $\mathcal N_C \to \tilde{\mathcal N}_C$ is an epimorphism.
 Consider the exact sequence
 $$0 \to\tilde{\mathcal L}^C\to\tilde{\mathcal M}^C\to\tilde{\mathcal M}_0^C\to\tilde{\mathcal L}_C\to
\tilde{\mathcal M}_C\to (\tilde{\mathcal M}_0)_C\to 0.$$
  Since    $\tilde{\mathcal M} \otimes \mathbb Q$ $\epi$ $\tilde{\mathcal M}_0 \otimes \mathbb Q$ is a split epimorphism,
  the image of $\tilde{\mathcal M}_0^C$ $\to$ $\tilde{\mathcal L}_C$ lies in the torsion subgroup, which is finite because of
  the epimorphism $\tilde{\mathcal N}_C\epi \tilde{\mathcal L}_C.$ Then the cokernel of $\tilde{\mathcal M}^C$ $\to$ $\tilde{\mathcal M}_0^C$ is finite.
  Set \begin{align*} & Q={\rm Coker}(\mathcal M_C \oplus \tilde{\mathcal M}^C \to \tilde{\mathcal M}_C),\\
  & Q_0={\rm Coker}((\mathcal M_0)_C \oplus \tilde{\mathcal M}_0^C \to (\tilde{\mathcal M}_0)_C).\end{align*}
  Then we know that $Q_0$ is finite and ${\rm Coker}(\tilde{\mathcal M}^C$ $\to$ $\tilde{\mathcal M}_0^C)$ is finite, and we need to prove that $Q$ is finite.
  Consider the diagram with exact columns.
$$
\xyma{
\mathcal N_C \oplus \tilde{\mathcal N}^C \ar@{->}[r] \ar@{->>}[d] & \mathcal M_C \oplus \tilde{\mathcal M}^C \ar@{->}[d] \ar@{->}[r]^\alpha &
(\mathcal M_0)_C \oplus \tilde{\mathcal M_0}^C  \ar@{->}[d]\\
 \tilde{\mathcal N}_C  \ar@{->}[r] & \tilde{\mathcal M}_C \ar@{->>}[d] \ar@{->}[r] & (\tilde{\mathcal M}_0)_C \ar@{->>}[d] \\
  & Q\ar@{->}[r] & Q_0}
$$
Using the snake lemma, we obtain that ${\rm Ker}(Q\to Q_0)={\rm
Coker}( \alpha).$ Since $Q_0$ is finite and ${\rm
Coker}(\alpha)={\rm Coker}(\tilde{\mathcal M}^C\to \tilde{\mathcal
M}_0^C)$ is finite, we get that $Q$ is finite.

\end{proof}

\begin{Corollary}\label{keycor} Let $M$  be a tame $\mathbb Z[C]$-module and $\mu_M=(x-1)^{m}f_1^{m_1}\dots f_l^{m_l} $ for some distinct monic irreducible polynomials
$f_1,\dots,f_l\in \mathbb Z[x]$ such that $f_i(1)\ne 0$. Assume
that for any $1\leq i\leq l$ either $f_i(0)\notin \{-1,1\}$ or
$m_i=1.$ Then the cokernel of the homomorphism
$$ (M^{\otimes 2})_C\oplus (\hat M^{\otimes 2})^C \longrightarrow ( \hat M^{\otimes 2} )_C$$
is finite.
\end{Corollary}

\begin{Remark} We prove Propositions  \ref{Proposition_omega} and \ref{onekey} for tensor products of some modules and their completions. Further we need the same statements for exterior squares. Of course, the statements for tensor products imply the statements for exterior squares, so it is enough to prove for tensor products. Moreover, it is more convenient to prove such statements for tensor products because they have two advantages.

The first obvious advantage is that we can change modules $M$ and $M'$ in the tensor product $M\otimes M'$ independently doing some reductions to `simpler' modules.

The second less obvious advantage is the following. Let $A$ be an abelian group and $M,M'$ are $\mathbb Z[A]$-modules. Then we can interpret coinvariants of the tensor product as the tensor product over $\mathbb Z[A]$
$$(M\otimes M')_A=M\otimes_{\mathbb Z[A]} M'_{\sigma},$$
where $M'_\sigma$ is the module $M'$ with the twisted module structure $m*a=ma^{-1}.$ In particular, there is an additional nontrivial structure of $\mathbb Z[A]$-module on $(M\otimes M')_A.$ But there is no such a structure on $(\Lambda^2 M)_A.$ More precisely, the kernel of the epimorphism $$(M\otimes M)_A \epi (\Lambda^2 M)_A$$ is {\bf not} always a $\mathbb Z[A]$-submodule. For example, if $A=C=\langle t \rangle,$ $M=\mathbb Z^2$ where $t$ acts on $M$ via the matrix $\left( \begin{smallmatrix} 1 & 1 \\ 0 & 1 \end{smallmatrix}\right),$ it is easy to check that the kernel is not a submodule.

In our article \cite{IvanovMikhailov} there are two  mistakes concerning this that can be fixed easily.
\begin{enumerate}
\item On page 562 we define $\wedge^2_\sigma M$ as a quotient module of $M\otimes_\Lambda M_\sigma$ by the submodule generated by the elements $m\otimes m.$ Then we prove Corollaries 3.4 and 3.5 for such a module. In the proof of Proposition 7.2 we assume that $\wedge^2_\sigma M=(\wedge^2 M)_A$ which is the first mistake. In order to fix this mistake we have define $\wedge^2_\sigma M$ as a quotient of $M\otimes_\Lambda M$ by the {\it abelian group} generated by the elements $m\otimes m$ and prove Corollaries 3.4 and 3.5 using this definition. The prove is the same. We just need to change the meaning of the word `generated' form `generated as a module' to 'generated as an abelian group'.

\item In the proof of Lemma 7.1 we assume that $(\wedge^2 M)_A$ is an $\mathbb Z[A]$-module. This is the second mistake. In order to fix it, we have replace $(\wedge^2 M)_A$ by $(M\otimes M)_A$ in the first sentence of the proof of Lemma 7.1.
\end{enumerate}
\end{Remark}

\section{$H\mathbb Z$-localization of $M\rtimes
C$}\label{hzsection}

Now consider our group $G=M\rtimes C$ and the maximal nilpotent
submodule $N\subseteq M$, such that $(M/N)_C$ is finite. We have a
natural commutative diagram
$$
\xyma{N\ar@{>->}[r] & G\ar@{->>}[r]\ar@{->}[d]^{\eta_\omega} & G/N\ar@{->}[d]^{\eta_\omega}\\
N\ar@{>->}[r] & \hat G\ar@{->>}[r] & \widehat{G/N}}
$$
Observe that, for any $\mathbb Z[C]$-submodule $N'\subseteq N$,
$$
\widehat{G/N'}=\hat G/N'.
$$
\begin{Lemma}\label{prlemma}
For any $\mathbb Z[C]$-submodule $N'$ of $N$, there is a natural
isomorphism
$$
H_2(\eta_\omega)(G)= H_2(\eta_\omega)(G/N').
$$
\end{Lemma}
\begin{proof}
We can present the submodule $N'$ as a finite tower of central
extensions. If we will prove that
$$
H_2(\eta_\omega)(G)=H_2(\eta_\omega)(G/N')
$$
for any $N'\subseteq N$ such that $N'(1-t)=0$, than we will be
able to prove the general statement by induction on class of
nilpotence of $N'$.

The assumption that $N'(1-t)=0$ implies that the extensions
$$
1\to N'\to G\to G/N'\to 1,\ \ \ 1\to N'\to \hat G\to
\widehat{G/N'}\to 1
$$
are central. Consider the natural map between sequences
(\ref{delta1}) for these extensions:
$$
\xyma{(G_{ab}\otimes N')/U \ar@{=}[d]\ar@{->}[r] &
H_2(G)\ar@{->}[r] \ar@{->}[d] &
H_2(G/N')\ar@{->}[r] \ar@{->}[d] & N' \ar@{=}[d] \ar@{->}[r] & H_1(G)\ar@{=}[d] \\
(\hat G_{ab}\otimes N')/U\ar@{->}[r] & H_2(\hat G)\ar@{->}[r]
\ar@{->>}[d] &
H_2(\widehat{G/N'})\ar@{->}[r] \ar@{->>}[d] & N'\ar@{->}[r] & H_1(\hat G)\\
& H_2(\eta_\omega)(G)\ar@{->}[r] & H_2(\eta_\omega)(G/N')}
$$
Elementary diagram chasing implies that the lower horizontal map
is an isomorphism and the needed statement follows.
\end{proof}
\begin{Lemma}\label{lemmao1}
If $E(G/N)=T_{\omega+1}(G/N)$, then $EG=T_{\omega+1}(G)$.
\end{Lemma}
\begin{proof}
First we observe that, for any $N'\subseteq N$, there is a natural
isomorphism
$$
T_{\omega+1}(G/N')=T_{\omega+1}(G)/N'.
$$
Indeed, lemma \ref{prlemma} implies that there is a natural
diagram
$$
\xyma{& N'\ar@{=}[r] \ar@{>->}[d] & N'\ar@{>->}[d]\\
H_2(\eta_\omega)(G)\ar@{=}[d] \ar@{>->}[r] & T_{\omega+1}(G)
\ar@{->>}[r]\ar@{->>}[d] & \hat G\ar@{->>}[d]\\
H_2(\eta_\omega)(G/N')\ar@{>->}[r]& T_{\omega+1}(G/N')\ar@{->>}[r]
& \widehat{G/N'}}
$$
Hence, we have a natural diagram
$$
\xyma{N'\ar@{>->}[r]\ar@{=}[d] & G\ar@{->>}[r]
\ar@{->}[d]^{\eta_{\omega+1}} &
G/N'\ar@{->}[d]^{\eta_{\omega+1}}\\ N'\ar@{>->}[r] &
T_{\omega+1}(G)\ar@{->>}[r] & T_{\omega+1}(G/N')}
$$
Again, as in the proof of lemma \ref{prlemma}, we will assume that
$N'$ is central and will prove that, in this case,
$EG=T_{\omega+1}(G)$ provided $E(G/N')=T_{\omega+1}(G/N')$.

This follows from comparison of sequences (\ref{delta1}) applied
to the above central extensions:
$$
\xyma{(G_{ab}\otimes N')/U \ar@{=}[d]\ar@{->}[r] &
H_2(G)\ar@{->}[r] \ar@{->}[d] &
H_2(G/N')\ar@{->}[r] \ar@{->}[d] & N' \ar@{=}[d] \ar@{->}[r] & H_1(G)\ar@{=}[d] \\
(\hat G_{ab}\otimes N')/U\ar@{->}[r] &
H_2(T_{\omega+1}(G))\ar@{->}[r] \ar@{->>}[d] &
H_2(T_{\omega+1}(G/N'))\ar@{->}[r] \ar@{->>}[d] & N'\ar@{->}[r] & H_1(T_{\omega+1}(G))\\
& H_2(\eta_{\omega+1})(G)\ar@{->}[r] & H_2(\eta_{\omega+1})(G/N')}
$$
Again, elementary diagram chasing shows that the lower horizontal
map is an isomorphism and the needed statement follows.
\end{proof}

\begin{Proposition}\label{keyprop}
For a tame $\mathbb Z[C]$-module $M$, the following conditions are
equivalent:\\
(i) $H\mathbb Z\text{-}{\sf length}(M\rtimes C)\leq \omega+1$;\\
(ii) the composition
$$
\Lambda^2(\hat M)^C\to \Lambda^2(\hat M)_C\to H_2(\eta_\omega)
$$
is an epimorphism.
\end{Proposition}

\begin{proof}
It follows from (\ref{delta}) and construction of
$T_{\omega+1}(G)$ that we have a natural diagram:
\begin{equation}\label{qdiag}
\xyma{& & H_2(G) \ar@{->}[d] \ar@{=}[r] & H_2(G)\ar@{->}[d]\\
H_3(\hat G) \ar@{->}[r] & G_{ab}\otimes H_2(\eta_\omega)(G)
\ar@{->}[r] & H_2(T_{\omega+1}(G)) \ar@{->}[r] & H_2(\hat
G)\ar@{->>}[d] \ar@{->>}[r] & H_2(\eta_\omega)(G)\ar@{=}[dl]\\&  &
& H_2(\eta_\omega)(G)}
\end{equation}
This diagram implies that the condition (i) for $G=M\rtimes C$ is
equivalent to the surjectivity of the map $\delta: H_3(\hat G)\to
G_{ab}\otimes H_2(\eta_\omega)(G).$ Proposition \ref{alpha33}
implies the following natural diagram
$$
\xyma{H_3(\hat G)\ar@{->}[rr] \ar@{->}[d] & & G_{ab}\otimes H_2(\eta_\omega)(G)\ar@{->>}[d]\\
\Lambda^2(\hat M)^C \ar@{->>}[rd] \ar@{->}[rr] & &
H_2(\eta_\omega)(G)\\ & \Lambda^2(\hat M)_C\ar@{->}[ru]}
$$
and the implication $(i) \Rightarrow (ii)$ follows.

Now assume that (ii) holds. Let $N$ be the maximal nilpotent
submodule of $M$ such that $(M/N)_C$ is finite. Denote
$H:=(M/N)\rtimes C$. We have a natural diagram:
$$ \xyma{H_3(\hat G)\ar@{->}[r] \ar@{->}[d] & \Lambda^2(\hat M)^C \ar@{->}[r] & \Lambda^2(\hat M)_C \ar@{->}[r] & H_2(\eta_\omega)(G)\ar@{->}[d] \\
H_3(\hat H)\ar@{->}[r] & \Lambda^2(\widehat{M/N})^C \ar@{->}[r] &
\Lambda^2(\widehat {M/N})_C \ar@{->}[r] & H_2(\eta_\omega)(H) }
$$
Lemma \ref{prlemma} implies that the right hand vertical map in
this diagram is a natural isomorphism. Condition (ii) implies that
the composition of three lower arrows in the last diagram must be
an epimorphism. Now observe that $H_{ab}\otimes
H_2(\eta_\omega)(H)=H_2(\eta_\omega)(H)$, since the group
$H_2(\eta_\omega)(H)$ is divisible and $(M/N)_C$ is finite.
Therefore, $\delta: H_3(\hat H)\to H_{ab}\otimes
H_2(\eta_\omega)(H)$ is surjective. The diagram (\ref{qdiag}) with
$G$ replaced by $H$ implies that $EH=T_{\omega+1}(H)$. Now the
statement (i) follows from lemma \ref{lemmao1}.
\end{proof}

Now we will consider the key example of a tame $\mathbb
Z[C]$-module $M$, such that $H\mathbb Z\text{-}{\sf
length}(M\rtimes C)>\omega+1$. For the construction of such an
example, recall first certain well-known properties of quadratic
functors.

Let $X_1,\dots,X_n, Y_1,\dots,Y_m$ be abelian groups and
$X=\bigoplus_{i=1}^n X_i,$ $Y=\bigoplus_{j=1}^m Y_j.$ An element
of a direct sum will be written as a column $(x_1,\dots,x_n)^{\sf
T}\in X$ and a homomorphism $f:X\to Y$ will be written as a matrix
$f=(f_{ji})$, where $f_{ji}:X_i\to Y_j$ and
$$f((x_1,\dots,x_n)^{\sf T})=(\sum_{i=1}^n f_{1i}(x_i), \dots ,
\sum_{i=1}^n f_{mi}(x_i))^{\sf T}.$$

For an abelian group $X$ we denote by $\Lambda^2 X, $ ${\sf S}^2
X$, $\Gamma^2 X$ and $X^{\otimes 2}$ its exterior, symmetric,
divided (the same as the Whitehead quadratic functor) and tensor
squares respectively. If $X$ is torsion free, then there are short
exact sequences
$$0 \xrightarrow{ \ \ \ } \Lambda^2 X \xrightarrow{ \ \iota_{\wedge} \ } X^{\otimes 2}
\xrightarrow{ \ \pi_{\sf S} \ }  {\sf S}^2 X \xrightarrow{ \ \ \ }
0, $$
$$0 \xrightarrow{ \ \ \ } \Gamma^2 X \xrightarrow{ \ \iota_{\Gamma} \ } X^{\otimes 2}
\xrightarrow{ \ \pi_{\wedge} \ }  \Lambda^2 X \xrightarrow{ \ \ \
} 0,
$$ where $\pi_\wedge, \pi_{\sf S}$ are the canonical projections
\begin{align*}& \iota_{\wedge}(x_1x_2)=x_1\otimes x_2-x_1\otimes
x_2,\\ & \iota_\Gamma(\gamma_2(x))=x\otimes x\end{align*} (see
\cite[ch.1 Section 4.3]{Illusie}). We will identify $\Gamma^2 X$
with the kernel of $\pi_\wedge$ for torsion free groups.

\begin{Lemma}\label{lemma6.1}
Let $M=\mathbb Z^2$ be the $\mathbb Z[C]$-module with the action
of $C$ given by the matrix $c=\left(\begin{smallmatrix} -1 & 1 \\
0 & -1
\end{smallmatrix} \right).$ Then $M(t-1)^{2n}=4^n M$ for any natural $n$ and $\hat M=(\mathbb Z_2)^2$ with the
action of $C$ given by the same matrix. Moreover,
\begin{align*} & \Lambda^2 \hat M=\Lambda^2 \mathbb Z_2\oplus \Lambda^2
\mathbb Z_2  \oplus \mathbb Z_2^{\otimes 2},\\ & (\Lambda^2 \hat
M)^C= \Lambda^2 \mathbb Z_2 \oplus 0 \oplus \Gamma^2 \mathbb Z_2,\\
& (\Lambda^2 \hat M)_C=0\oplus \Lambda^2 \mathbb Z_2\oplus {\sf
S}^2 \mathbb Z_2\end{align*} and the cokernel of the natural map
$$
(\Lambda^2 \hat M)^C \longrightarrow (\Lambda^2 \hat M)_C
$$
is isomorphic $ \Lambda^2 \mathbb Z_2.$
\end{Lemma}
\begin{proof}
Set $d:=c-1=\left(\begin{smallmatrix} -2 & 1 \\ 0 & -2
\end{smallmatrix} \right).$  Since  $MI^n=M(t-1)^n=d^n(\mathbb
Z^2),$ we have $\hat M=\varprojlim\: \mathbb Z^2/d^n(\mathbb
Z^2).$ Computations show that $d^{2}=-4 c,$ and hence
$d^{2n}=(-4)^n c^n.$ Since $c$ induces an automorphism on $\mathbb
Z^2,$ we obtain $d^{2n}(\mathbb Z^2)=4^n  \mathbb Z^2.$ Thus the
filtration $d^n(\mathbb Z^2)$ of $\mathbb Z^2$ is equivalent to
the filtration $2^n\mathbb Z^2.$ It follows that $\hat M=(\mathbb
Z_2)^2$ with the action of $C$ given by $c$. Then $\Lambda^2 \hat
M\cong (\Lambda^2 \mathbb Z_2)^2\oplus \mathbb Z_2^{\otimes 2}.$
We identify the element $$xe_1 \wedge x'e_1 + ye_2 \wedge y'e_2 +
ze_1\wedge z'e_2 \in \wedge^2 \hat M$$ with the column $$(x\wedge
x',y\wedge y', z\otimes z')^{\sf T} \in (\wedge^2 \mathbb
Z_2)^2\oplus (\mathbb Z_2 \otimes \mathbb Z_2),$$ where $e_1,e_2$
is the standard basis of $(\mathbb Z_2)^2$ over $\mathbb Z_2.$ Let
us present the homomorphism $\wedge^2 \hat c: \Lambda^2 \hat M \to
\Lambda^2 \hat M$ that defines the action of $C$ as a matrix.
Since
\begin{equation*}
\begin{split}
\wedge^2 \hat c(xe_1\wedge x'e_1) &=xe_1\wedge x'e_1 \\
 \wedge^2 \hat c(ye_2 \wedge y'e_2) &=ye_1\wedge y'e_1+ye_2 \wedge y'e_2-(ye_1\wedge y'e_2 -y'e_1\wedge ye_2) \\
\wedge^2 \hat c(ze_1\wedge z'e_2)&=-ze_1\wedge z'e_1+ze_1\wedge
z'e_2,
\end{split}
\end{equation*}
we obtain
$$ \wedge^2 \hat c=
\left(\begin{smallmatrix}
1 & 1 & -\pi_\wedge\\
0 & 1 & 0 \\
0 & -\iota_\wedge & 1
\end{smallmatrix} \right), \hspace{1cm} \wedge^2 \hat c-1=
\left(\begin{smallmatrix}
0 & 1 & -\pi_\wedge\\
0 & 0 & 0 \\
0 & -\iota_\wedge & 0
\end{smallmatrix} \right).
$$
Now it is easy to compute that $$(\Lambda^2 \hat M)^C= {\rm
Ker}(\wedge^2 \hat c-1)=\Lambda^2 \mathbb Z_2 \oplus 0 \oplus
\Gamma^2 \mathbb Z_2$$ and $${\rm Im}(\wedge^2 \hat c-1)=\Lambda^2
\mathbb Z_2 \oplus 0 \oplus \iota_\wedge(\Lambda^2 \mathbb Z_2).$$
It follows that $(\Lambda^2 \hat M)_C={\rm Coker}(\wedge^2 \hat
c-1)=0\oplus \Lambda^2 \mathbb Z_2\oplus {\sf S}^2 \mathbb Z_2.$
In order to prove that $${\rm Coker}((\Lambda^2 \hat M)^C \to
(\Lambda^2 \hat M)_C)=\Lambda^2 \mathbb Z_2,$$ we have to prove
that ${\rm Coker}(\Gamma^2 \mathbb Z_2 \to \mathbb Z_2^{\otimes 2}
\to {\sf S}^2 \mathbb Z_2)=0.$ In other words we have to prove
that $\mathbb Z_2^{\otimes 2}$ is generated by elements $x\otimes
x$ and $x\otimes y-y\otimes x$ for $x,y\in \mathbb Z_2.$ Fist note
that any $2$-divisible element is generated by elements $$2
x\otimes y=(x\otimes y-y\otimes x)+(x+y)\otimes (x+y)-x\otimes
x-y\otimes y.$$ Since any element of $\mathbb Z_2$ is equal to
$2x$ or $1+2x$ for some $x\in \mathbb Z_2,$ all other elements of
$\mathbb Z_2 \otimes \mathbb Z_2$ can be presented as sums of
elements $(1+2x)\otimes (1+2y)=1\otimes 1+2(x\otimes 1+1\otimes
y+2x\otimes y).$
\end{proof}

\begin{Theorem}\label{theorem6}
Let $M$ be the $\mathbb Z[C]$-module from lemma \ref{lemma6.1} The
$H\mathbb Z$-length of the group
$$G:=M\rtimes C=\langle a,b,t\ |\ [a,b]=1,\ a^t=a^{-1}, b^t=ab^{-1}\rangle$$
is greater than $\omega+1$.
\end{Theorem}
\begin{proof}
Consider the central sequence
$$
1\to H_2(\eta_\omega)(G)\to T_{\omega+1}(G)\to \hat G\to 1
$$
We will show that the cokernel of the map $\delta: H_3(\hat G)\to
G_{ab}\otimes H_2(\eta_\omega)(G)$ from (\ref{stem}) contains
$\Lambda^2(\mathbb Z_2)$. The theorem will immediately follow.

By proposition \ref{alpha33}, the map
$$
\delta: H_3(\hat G) \to G_{ab}\otimes
H_2(\eta_\omega)(G)=H_2(\eta_\omega)(G)
$$
factors as
$$
H_3(\hat G)\to (\Lambda^2\hat M)^C\to (\Lambda^2\hat
M)_C\twoheadrightarrow H_2(\eta_\omega)(G)
$$
The direct summand $\Lambda^2\mathbb Z_2$ from $(\Lambda^2\hat
M)_C$ maps isomorphically to a direct summand of
$H_2(\eta_\omega)(G)=\Lambda^2\mathbb Z_2\oplus ({\sf S}^2\mathbb
Z_2)/\mathbb Z.$ By lemma \ref{lemma6.1}, the summand
$\Lambda^2\mathbb Z_2$ lies in the cokernel of the map
$(\Lambda^2\hat M)^C\to (\Lambda^2\hat M)_C$, therefore, it lies
also in the cokernel of the composite map $(\Lambda^2\hat M)^C\to
H_2(\eta_\omega)(G)$ as well as of the map $\delta$.
\end{proof}

\begin{Theorem}\label{Theorem_omega+1} Let $G$ be a metabelian group of the form $G=M\rtimes C,$ where
$M$ is a tame $C$-module and
$\mu_M=(x-\lambda_1)^{m_1}\dots (x-\lambda_l)^{m_l}$ for some
distinct complex numbers $\lambda_1,\dots,\lambda_l$ and $m_i\geq 1.$
\begin{enumerate}
\item Assume that the equality $\lambda_i\lambda_j=1$ holds only
if  $\lambda_i=\lambda_j=1.$ Then $$H\mathbb Z\text{-}{\sf
length}(G)\leq \omega.$$ \item Assume that the equality
$\lambda_i\lambda_j=1$ holds only if either $m_i=m_j=1$ or
$\lambda_i=\lambda_j=1.$ Then $$H\mathbb Z\text{-}{\sf
length}(G)\leq \omega+1.$$
\end{enumerate}

\end{Theorem}
\begin{proof}
Follows from propositions \ref{keyprop}, \ref{Proposition_omega}
and \ref{onekey}.
\end{proof}

\begin{Lemma}\label{Lemma_second_homology} Let $M$ be a tame $\mathbb Z[C]$-module such that $H_2(M\rtimes C)$ is finite and
$$\mu_M=(x-\lambda_1)^{m_1}\dots (x-\lambda_l)^{m_l},$$ where
$\lambda_1,\dots,\lambda_l$ are distinct complex numbers and
$m_i\geq 1$. Then  $\lambda_i\lambda_j=1$ implies $m_i=m_j=1.$
\end{Lemma}
\begin{proof}
Set $V:=M\otimes \mathbb C.$ Then by Lemma
\ref{Lemma_structure_theorem_fields} we get $V=V_1\oplus \dots
\oplus V_l$ such that $\mu_{V_i}=(x-\lambda_i)^{m_i}.$ The short
exact sequence $$(\Lambda^2 M)_C \mono H_2(M\rtimes C) \epi M^C$$
implies that $(\wedge^2 M)_C$ is finite. Then $(\Lambda^2_{\mathbb
C} V)_C= (\Lambda^2 M)_C\otimes \mathbb C=0.$ Assume the contrary,
that there exist $i,j$ such that $\lambda_i\lambda_j=1$ and
$m_i\geq 2.$ Consider two cases $i=j$ and $i\ne j.$

(1) Assume that $i=j.$ Then $\lambda_i=\lambda_j=-1.$ Since
$m_i\geq 2,$ at least one of Jordan blocks of $a_M\otimes \mathbb
C$ corresponding to $-1$ has size bigger than $1.$ It follows that
there is a epimorphism $V\epi U,$ where  $U= \mathbb C^2$ and $C$
acts on $U$ by the matrix $\left( \begin{smallmatrix} -1 & \ 1 \\
\ 0 & -1 \end{smallmatrix} \right).$ The epimorphism induces a
epimorphism  $0=(\Lambda^2_{\mathbb C} V)_C \epi (\wedge_{\mathbb
C}^2 U)_C.$ From the other hand, a simple computation shows that
$(\Lambda_{\mathbb C}^2 U)_C\cong \mathbb C.$ So we get a
contradiction.

(2) Assume that $i\ne j.$ Then $\lambda_i\ne \lambda_j.$ Because
of the isomorphism of $\mathbb C[C]$-modules $$\Lambda^2_{\mathbb
C} V=(\bigoplus_k \Lambda^2_{\mathbb C} V_k) \oplus
(\bigoplus_{k<k'} V_k \otimes_{\mathbb C} V_{k'}),$$ we have an
epimorphism of $\mathbb Z[C]$-modules $\Lambda^2_{\mathbb C} V
\epi V_i\otimes_{\mathbb C} V_j.$ Since $m_i\geq 2,$ there is an
epimorphism $V_i\epi U_i,$ where $U_i=\mathbb C^2$ and $C$ acts on
$U_i$ by the matrix $\left( \begin{smallmatrix} \lambda_i &  1 \\
0 & \lambda_i
\end{smallmatrix} \right).$ Moreover, there is an epimorphism
$V_j\epi U_j,$ where $U_j=\mathbb C$ and $C$ acts on $U_j$ by the
multiplication on $\lambda_j.$ It follows that $C$ acts on
$U_i\otimes_{\mathbb C} U_j$ by the matrix $\left(
\begin{smallmatrix} 1 &  \lambda_j \\ \ 0 & 1 \end{smallmatrix}
\right).$ Thus $(U_i\otimes_{\mathbb C} U_j)_C\cong \mathbb C.$
From the other hand we have an epimorphism $0=(\Lambda_{\mathbb
C}^2 V)_C\epi (U_i\otimes U_j)_C.$ So we get a contradiction.
\end{proof}

\begin{Proposition}\label{h2pr}
Let $G$ be a metabelian finitely presented group of the form
$G=M\rtimes C$ for some $\mathbb Z[C]$-module $M$ and $H_2(G)$ is
finite. Then $$H\mathbb Z\text{-}{\sf length}(G)\leq \omega+1.$$
\end{Proposition}
\begin{proof}
It follows from Lemma \ref{Lemma_second_homology} and Theorem
\ref{Theorem_omega+1}.
\end{proof}

\section{Bousfield's method}\label{boussection}
Let $G$ be a finitely presented group given by presentation
$$
\langle x_1,\dots, x_m\ |\ r_1,\dots, r_k\rangle
$$
Consider the free group $F=F(x_1,\dots, x_n)$ and an epimorphism
$F\to G$ with the kernel normally generated by $k$ elements
$\ker{F\to G}=\langle r_1,\dots, r_k\rangle^F$. Here we will study
the induced map
\begin{equation}\label{hr}
h: H_2(\hat F)\to H_2(\hat G).
\end{equation}
We follow the scheme due to Bousfield from \cite{Bousfield}. Let
$$
F_*:\ \ \ \ \ldots \begin{matrix}\longrightarrow\\[-3.5mm]\longrightarrow\\[-3.5mm]\longrightarrow\\[-3.5mm]\longleftarrow\\[-3.5mm]
\longleftarrow
\end{matrix}\ F_1\ \begin{matrix}\buildrel{d_0,d_1}\over\longrightarrow\\[-3.5mm]\longrightarrow\\[-3.5mm]
\longleftarrow \end{matrix}\ F_0(=F)\to G$$ be a free simplicial
resolution of $G$, where $F_1$ a free group with $m+k$ generators.
The structure of $F_1$ is $F(y_1,\dots, y_k)*F$, and the maps
$d_0,d_1$ are given as
\begin{align*}
& d_0: y_i\mapsto 1,\ F\buildrel{id}\over\mapsto F\\
& d_1: y_i\mapsto r_i,\ F\buildrel{id}\over\mapsto F
\end{align*}
The following short exact sequence follows from Lemma 5.4 and
Proposition 3.13 in \cite{Bousfield}:
\begin{equation}\label{lim1term} 1\to \ilimit^1_k
\pi_1(F_*/\gamma_k(F_*))\to \pi_0( \widehat{F_*})\to \hat G\to 1
\end{equation}
The first homotopy group $\pi_1(F_*/\gamma_k(F_*))$ is isomorphic
to the $k$th Baer invariant known also as $k$-nilpotent
multiplicator of $G$ (see \cite{BE}, \cite{Ellis}). If $G=F/R$ for
a normal subgroup $R\lhd F$, the Baer invariant can be presented
as the quotient
$$
\pi_1(F_*/\gamma_*(F_*))\simeq \frac{R\cap
\gamma_k(F)}{[R,\underbrace{F,\dots, F}_{k-1}]}.
$$
Now assume that, for a group $G$, the $\ilimit^1$-term vanishes in
(\ref{lim1term}), that is, there is a natural isomorphism
$$
\pi_0(\widehat{F_*})\simeq \hat G
$$
There is the first quadrant spectral sequence (see \cite{BK}, page
108)
\begin{align*}
& E_{p,q}^1=H_q(\widehat{F_p})\Rightarrow H_{p+q}(\overline{W}
(\widehat{F_*}))\\
& d_r: E_{p,q}^r\to E_{p-r, q+r-1}^r
\end{align*}
As a result of convergence of this spectral sequence, we have the
following diagram:
\begin{equation}\label{einfty}
\xyma{E_{0,1}^1=H_2(\hat F)\ar@{->>}[r] & E_{0,1}^\infty
\ar@{>->}[d]
\\ & H_2(\overline{W}(\widehat{F_*})) \ar@{->>}[r] \ar@{->>}[d] & H_2(\hat G) \\ & E_{1,0}^\infty}
\end{equation}
Since $F_1$ is finitely generated, $H_1(\widehat{F_1})\simeq
H_1(F_1)$, therefore, $E_{1,0}^\infty$ is a quotient of
$\pi_1((F_*)_{ab})=H_2(G)$. Now the diagram (\ref{einfty}) implies
the following
\begin{Lemma}\label{lemmah2}
Assuming that $\ilimit^1$-term in (\ref{lim1term}) vanishes, the
cokernel of the map (\ref{hr}) is isomorphic to a quotient of the
homology group $H_2(G)$. For a group $G$ with $H_2(G)=0$, the map
$h: H_2(\hat F)\to H_2(\hat G)$ is an epimorphism.
\end{Lemma}

\section{Main construction}\label{mainsection}

\begin{Lemma}\label{newlemma} Let $K=\langle a,b\mid [[a,b],a]=[[a,b],b]=1 \rangle$ be the free 2-generated nilpotent group of class $2$ and
$C=\langle t \rangle$ acts on $K$ by the following automorphism
\begin{align*} & a\mapsto a^{-1}\\ & b\mapsto ab^{-1}.\end{align*}
Set $\tilde K:=\varprojlim K/\gamma_n(K\rtimes C).$ Then the
following holds
\begin{enumerate}
\item the pronilpotent completion of $K \rtimes C$ is equal to
$\tilde K\rtimes C;$ \item $K_{\sf ab}$ is isomorphic to the
$\mathbb Z[C]$-module from Lemma \ref{lemma6.1}; \item the obvious
map $(\tilde K)_{\sf ab}\xrightarrow{\cong}\widehat{(K_{\sf ab})}$
is an isomorphism of $\mathbb Z[C]$-modules; \item either $[\tilde
K,\tilde K]\cong \mathbb Z_2$ or $[\tilde K,\tilde K]\cong \mathbb
Z/2^m$ for some $m.$
\end{enumerate}
\end{Lemma}
\begin{proof} For the sake of simplicity we set $K_n:=\gamma_n(K\rtimes C).$
(1) follows from the equality $$(K\rtimes C)/K_n=(K/K_n)\rtimes
C$$ for $n\geq 2.$ (2) is obvious. Prove (3) and (4).

It is obvious that $\gamma_2(K)=Z(K)=\{[a,b]^k\mid k\in \mathbb
Z\}\cong \mathbb Z$ and $[a^i,b^j]=[a,b]^{ij}$ for all $i,j\in
\mathbb Z.$ Moreover, any element of $K$ can be uniquely presented
as $a^ib^j[a,b]^k$ for $i,j,k\in \mathbb Z.$ Set $M=K_{\sf ab}.$
We consider $M$ with the additive notation as a module over $C$.
Then $\gamma_{n+1}(M\rtimes C)=M(t-1)^n.$ Lemma \ref{lemma6.1}
implies that $\gamma_{2n+1}(M\rtimes C)=4^n M.$ Since the map
$K_{2n+1}\epi \gamma_{2n+1}(K\rtimes C)$ is an epimorphism, there
exist $k,l\in \mathbb Z$ (that depend of $n$) such that
$a^{4^n}[a,b]^k,b^{4^n}[a,b]^l\in K_{2n+1}.$ Since,
$[a,b]^{4^{2n}}=[a^{4^n}[a,b]^k,b^{4^n}[a,b]^l],$ we obtain
$[a,b]^{4^{2n}} \in K_{2n+1}.$ Then $K_{2n+1} \cap
\gamma_2(K)=\langle [a,b]^{2^{m(2n+1)}}
 \rangle$ for some natural number $m(2n+1)$ because $4^{2n}$ is divisible only on powers of $2.$ Hence $K_{n} \cap \gamma_2(K)=\langle [a,b]^{2^{m(n)}}
 \rangle$   for some $m(n).$ Then the short exact sequence $\gamma_2(K)\mono K \epi M$ induces the short exact sequence
 $$0 \longrightarrow \mathbb Z/2^{m(n)} \xrightarrow{\ \ \iota_n\ \ } K/ K_n \longrightarrow M/M(t-1)^n \longrightarrow 1,$$
 where $\iota_n(1)=[a,b].$ Since $\mathbb Z/2^{m(n)}$ and $M/M(t-1)^n$ are finite $2$-groups, the order of $K/K_n$ is equal to $2^{m'(n)}$ for some $m'(n).$
  We obtain a short exact sequence
$$0 \longrightarrow\: \varprojlim \mathbb Z/2^{m(n)} \xrightarrow{\ \ \hat{\iota}\ \ } \tilde K \longrightarrow \hat M \longrightarrow 1.$$
If $m(n)\to \infty,$ then $\varprojlim \mathbb Z/2^{m(n)}=\mathbb
Z_2,$ else $m(n)$ stabilizes and $\varprojlim\: \mathbb
Z/2^{m(n)}=\mathbb Z/2^m.$

Now it is sufficient to prove that $[\tilde K,\tilde K]={\rm
Im}(\hat \iota).$ Since $\hat M$ is abelian, $[\tilde K,\tilde
K]\subseteq {\rm Im}(\hat \iota).$  Prove that $[\tilde K,\tilde
K]\supseteq {\rm Im}(\hat \iota).$ Any element of $\tilde K$ can
be presented as a sequence $(x_n)_{n=1}^{\infty},$ where $x_n \in
K/K_n$ such that $x_n \equiv x_{n+1} \ {\rm mod}\ K_n.$ Any
element of $\varprojlim \mathbb Z/2^{m(n)}$ can be presented as an
image of a $2$-adic integer $\sum_{k=0}^\infty \alpha_k 2^k,$
where $\alpha_k\in \{0,1\}.$
 Then an element of  ${\rm Im}(\hat \iota)$ can be presented as a sequence
 $$([a,b]^{\sum_{k=0}^{m(n)} \alpha_k2^k} )_{n=1}^{\infty}=([a^{\sum_{k=0}^{m(n)}
 \alpha_k2^k},b])_{n=1}^{\infty}.$$
 Note that the element $(a^{\sum_{k=0}^{m'(n)} \alpha_k2^k})_{n=1}^{\infty}$ is a well defined element of $\tilde K$ because
 $a^{2^{m'(n)}}\in K_n.$ It follows that
 $$([a,b]^{\sum_{k=0}^{m(n)} \alpha_k2^k} )_{n=1}^{\infty}=[(a^{\sum_{k=0}^{m(n)}
 \alpha_k2^k})_{n=1}^{\infty},(b)_{n=1}^{\infty}],$$
 and hence, $[\tilde K,\tilde K]\supseteq {\rm Im}(\hat \iota).$
\end{proof}

\begin{Theorem}
Let $F$ be a free group of rank $\geq 2$. Then $H\mathbb
Z\text{-}{\sf length}(F)\geq \omega+2$.
\end{Theorem}
\begin{proof}
Consider the following group:
$$
\Gamma:=\langle a,b,t\ |\ [[a,b],a]=[[b,a],b]=1,\ a^t=a^{-1},\
b^t=ab^{-1}\rangle
$$
Observe that $\Gamma$ is the semidirect product $K\rtimes C$ from
the previous lemma.

Consider the natural diagram induced by the projection $K\to
K_{ab}=M$, i.e. by $\Gamma=K\rtimes C\twoheadrightarrow G=M\rtimes
C$:
\begin{equation}\label{diav} \xyma{H_3(\hat \Gamma)\ar@{->}[r]^{\delta\
\ \ \ \ \ \ \ \ }
\ar@{->}[d] & H_1(\Gamma)\otimes H_2(\eta_\omega)(\Gamma)\ar@{->}[d]\ar@{-->}[rd]\\
H_3(\hat G)\ar@{->}[r]^{\delta\ \ \ \ \ \ \ \ } & H_1(G)\otimes
H_2(\eta_\omega)(G)\ar@{->>}[r] & {\rm Coker}(\delta)(G)}
\end{equation}

It is shown in the proof of theorem \ref{theorem6} that ${\rm
Coker}(\delta)(G)$ contains $\Lambda^2\mathbb Z_2$. This term
$\Lambda^2\mathbb Z_2$ is an epimorphic image of one of the terms
$\Lambda^2\mathbb Z_2$ in
$$
H_2(\hat M)=\Lambda^2\hat M=\Lambda^2\mathbb Z_2\oplus
\Lambda^2\mathbb Z_2\oplus \mathbb Z_2\otimes \mathbb Z_2
$$
(see lemma \ref{lemma6.1}). By lemma \ref{newlemma}, there is a
short exact sequence
$$
H_2(\tilde K)\to H_2(\hat M)\twoheadrightarrow [\tilde K, \tilde
K]
$$
Now, by lemma \ref{newlemma}, $[\tilde K, \tilde K]$ is either
$\mathbb Z_2$ or a finite group, therefore, the image of both
terms $\Lambda^2\mathbb Z_2$ in $[\tilde K, \tilde K]$ are zero
(2-adic integers do not contain divisible subgroups). In
particular, the term $\Lambda^2\mathbb Z_2$ which maps
isomorphically onto a subgroup of ${\rm Coker}(\delta)(G)$ lies in
the image of $H_2(\tilde K)$. The natural square
$$
\xyma{H_2(\tilde K)\ar@{->>}[r]\ar@{->}[d] &
H_2(\eta_\omega)(\Gamma)\ar@{->}[d]\\ H_2(\hat M)\ar@{->>}[r] &
H_2(\eta_\omega)(G)}
$$
is commutative and  we conclude that the diagonal arrow in
(\ref{diav}) maps onto the subgroup $\Lambda^2\mathbb Z_2$ of
${\rm Coker}(\delta)(G)$. Hence, ${\rm Coker}(\delta)(\Gamma)$
maps epimorphically onto $\Lambda^2\mathbb Z_2$.

Now lets prove that the second homology $H_2(\Gamma)$ is finite.
Since the group $\Gamma$ is the semi-direct product $K\rtimes C$,
its homology is given as
$$
0\to H_2(K)_C\to H_2(\Gamma)\to H_1(C, K_{ab})\to 0
$$
The right hand term is zero: $H_1(C, K_{ab})=H_1(C,M)=M^C=0$. It
follows immediately that
$H_2(K)_C=(\gamma_3(F(a,b))/\gamma_4(F(a,b)))_C=\mathbb Z/4$.

Observe that the group $\Gamma$ can be defined with two generators
only. Let $F$ be a free group of rank $\geq 2$. Consider a free
simplicial resolution of $\Gamma$ with $F_0=F$: $F_*\to \Gamma$.
Since $H_2(\Gamma)$ is finite, all Baer invariants of $\Gamma$ are
finite (see, for example, \cite{Ellis})
$$
\ilimit^1 \pi_1(F_*/\gamma_k(F_*))=0
$$
and, by lemma \ref{lemmah2}, the cokernel of the natural map
$H_2(\hat F)\to H_2(\hat \Gamma)$ is finite.

We have a natural commutative diagram:
\begin{equation}\label{mapd} \xyma{H_3(\hat F)\ar@{->}[r]^{\delta\
\ \ \ \ \ \ \ \ }\ar@{->}[d] & H_1(F)\otimes
H_2(\eta_\omega)(F)\ar@{->}[d]\ar@{->>}[r] & {\rm Coker}(\delta)(F)\ar@{->}[d]\\
H_3(\hat \Gamma)\ar@{->}[r]^{\delta\ \ \ \ \ \ \ \ \ } &
H_1(\Gamma)\otimes H_2(\eta_\omega)(\Gamma)\ar@{->>}[r] & {\rm
Coker}(\delta)(\Gamma)}
\end{equation}
with $H_2(\eta_\omega)(F)=H_2(\hat F)$. Since the cokernel of
$H_2(\hat F)\to H_2(\hat \Gamma)$ is finite,
$$
{\rm Coker}\{{\rm Coker}(\delta)(F)\to {\rm
Coker}(\delta)(\Gamma)\}
$$
also is finite. However, ${\rm Coker}(\delta)(\Gamma)$ maps onto
$\Lambda^2\mathbb Z_2$, as we saw before, hence ${\rm
Coker}(\delta)(F)$ is uncountable. Therefore, by (\ref{delta}),
$H_2(T_{\omega+1}(F))\neq 0$ and the statement is proved.
\end{proof}

\section{Alternative approaches}
In general, given a group, the description of its pro-nilpotent
completion is a difficult problem. If a group is not
pre-nilpotent, its pro-nilpotent is uncountable and may contain
complicated subgroups. In this paper, as well as in
\cite{IvanovMikhailov} we essentially used the explicit structure
of pro-nilpotent completion for metabelian groups. Now we observe
that, there is a trick which gives a way to show that some groups
have $H\mathbb Z$-length greater than $\omega$ without explicit
description of their pro-nilpotent completion. In a sense, the
Bousfield scheme described above also gives such a method, however
in that way one must compare the considered group with a group
with clear completion. The trick given bellow is different.

We put $\Phi_iH_2(G)={\rm Ker}\{H_2(G)\to
H_2(G/\gamma_{i+1}(G))\}.$ Then $\Phi_iH_2(G)$ is the Dwyer
filtration on $H_2(G)$ (see \cite{Dw}).

\begin{Proposition}
Let $G$ be a group with the following properties:\\ (i)
$\gamma_\omega(G)\neq\gamma_{\omega+1}(G)$\\
(ii) $\cap_i\Phi_iH_2(G/\gamma_\omega(G))=0$.\\
Then the $H\mathbb Z$-length of $G$ is greater than $\omega$.
\end{Proposition}
\begin{proof}
All ingredients of the proof are in the following diagram with
exact horizontal and vertical sequences:
$$
\xyma{& H_2(G) \ar@{->}[d] \ar@{->}[rd]\\
ker\ar@{>->}[r]\ar@{>->}[d] &
H_2(G/\gamma_\omega(G))\ar@{->}[r] \ar@{->>}[d] & H_2(\hat G)\\
\cap_i\Phi_iH_2(G/\gamma_\omega(G)) &
\gamma_\omega(G)/\gamma_{\omega+1}(G)}
$$
The fact that the kernel lies in
$\cap_i\Phi_iH_2(G/\gamma_\omega(G))$ follows from the standard
property of Dwyer's filtration: the map $G/\gamma_\omega(G)\to
\hat G$ induces isomorphisms on $H_2/\Phi_i$ for all $i$ (see
\cite{Dw}). The vertical exact sequence is the part of 5-term
sequence in homology. Now assume that the map $H_2(G)\to H_2(\hat
G)$ is an epimorphism. Then $H_2(G/\gamma_\omega(G))\to H_2(\hat
G)$ is an epimorphism as well. However, condition (ii) implies
that the last map is a monomorphism as well. Hence, it is an
isomorphism and surjectivity of $H_2(G)\to H_2(\hat G)$
contradicts the property (i).
\end{proof}

An example of a group with satisfies both conditions (i) and (ii)
from proposition, is the group
$$
\langle a,b\ |\ [a,b^3]=[[a,b],a]^2=1\rangle
$$
(see \cite{MikhailovPassi}, examples 1.70 and 1.85). However, in
this example, $G/\gamma_\omega(G)$ is metabelian and one can use
explicit description of its pro-nilpotent completion to show the
same result. The above proposition can be used for more
complicated groups.

 Now we
consider another example of finitely generated metabelian group of
the type $M\rtimes C$ with $H\mathbb Z$-length greater than
$\omega+1$. In our example, $M=\mathbb Z[C]$ with a regular action
of $C$ as multiplication. This is not a tame $\mathbb Z[C]$-module
and the group is not finitely presented. The proof of the bellow
result uses functorial arguments.
\begin{Theorem}\label{zwrz}
Let $$G=\mathbb Z[C]\rtimes C=\mathbb Z\wr\mathbb Z=\langle a,b\
|\ [a,a^{b^i}]=1,\ i\in \mathbb Z\rangle,$$ then $H\mathbb
Z\text{-}{\sf length}(G)\geq \omega+2$.
\end{Theorem}
\begin{proof}
We define a functor from the category ${\sf fAb}$ of finitely
generated free abelian groups to the category of groups as
follows:
$$
\mathcal F: A\mapsto (A\otimes M)\rtimes C,\ A\in {\sf fAb},
$$
where the action of $C$ on $A$ is trivial and $M=\mathbb Z[C]$.
Now our main example can be written as $G=\mathcal F(\mathbb Z)$.
Since the action of $C$ on $A$ is trivial, the pro-nilpotent
completion of $\mathcal F(A)$ can be described as follows:
$$
\widehat{\mathcal F(A)}=(A\otimes \hat M)\rtimes C.
$$
One can easily see that
$$
H_1(C, H_i(A\otimes \hat M))=0, i\geq 1.
$$
Since the abelian group $A\otimes \hat M$ is torsion-free,
$$
H_i(\widehat{\mathcal F(A)})= \Lambda^i(A\otimes \hat M)_C,\ i\geq
1.
$$
Now we have
$$
H_2(\eta_\omega)(\mathcal F(A))={\rm Coker}\{\Lambda^2(A\otimes
M)_C\to \Lambda^2(A\otimes \hat M)_C\}.
$$
Now consider the functor from the category of finitely generated
free abelian groups to all abelian groups
$$
\mathcal G: A\mapsto H_2(\eta_\omega)(\mathcal F(A))
$$
It follows immediately that $\mathcal G$ is a quadratic functor.
Indeed, it is a proper quotient of the functor $A\mapsto A\otimes
A\otimes B$ for a fixed abelian group $B$.

The exact sequence (\ref{delta}) applied to the stem-extension
$$
1\to
H_2(\eta_\omega)(\mathcal F(A))\to T_{\omega+1}(\mathcal F(A))\to
\widehat{\mathcal F(A)}\to 1
$$
can be viewed as an exact sequence of functors in the category
${\sf sAb}$. Consider the map $\delta$ as a natural transformation
in ${\sf sAb}$. The functor $H_3(\widehat{\mathcal
F(A)})=\Lambda^3(A\otimes \hat M)_C$ is a quotient of the cubic
functor
$$
A\otimes A\otimes A\otimes D
$$
for some fixed abelian group $D$ (which equals to the tensor cube
of $\hat M$). Recall that any natural transformation of the form
$$
A\otimes A\otimes A\otimes D\to {\sf quadratic\ functor}
$$
is zero. This follows from the simple observation that $A\otimes
A\otimes A\otimes D$ is a natural epimorphic image of its triple
cross-effect. See, for example, \cite{Mikh} for generalizations
and detailed discussion of this observation.

Therefore, for any non-zero $A$, the map $\delta$ in (\ref{delta})
is zero and $H_2(T_{\omega+1}(\mathcal F(A)))$ contains a subgroup
$H_2(\eta_{\omega})(\mathcal F(A))$ which is uncountable. In
particular, $H_2(T_{\omega+1}(G))=H_2(T_{\omega+1}(\mathcal
F(\mathbb Z))$ is uncountable and hence $H_2(G)\to
H_2(T_{\omega+1}(G))$ is not an epimorphism.
\end{proof}

\begin{Remark}
If we change the group $\mathbb Z\wr \mathbb Z$ by lamplighter
group by adding one more relation, we will obtain another wreath
product (for $n\geq 2$)
$$
\mathbb Z/n\wr \mathbb Z=\langle a,b\ |\ a^n=1, [a,a^{b^i}]=1,\
i\in \mathbb Z\rangle
$$
One can use the scheme of this paper to prove that $H\mathbb
Z\text{-}{\sf length}(\mathbb Z/n\wr \mathbb Z)>\omega+1$.
Essentially it follows from the triviality of $\Lambda^2(\hat
M)^C$ as in Theorem \ref{zwrz}. As it was shown by G. Baumslag
\cite{Baumslag}, there are different ways to embed this group into
a finitely-presented metabelian group. For example, $\mathbb
Z/n\wr Z$ is a subgroup generated by $a,b$ in
$$G_{[n]}:=\langle a,b,c\ |\ a^n=[a,a^b]=[a,a^{b^2}]=[b,c]=1,
a^c=aa^ba^{-b^2}\rangle.$$ Now we have a decomposition
$G_{[n]}=(\mathbb Z/n)^{\oplus \infty}\rtimes (C\times C)$ and it
follows immediately from \cite{IvanovMikhailov} that $H\mathbb
Z\text{-}{\sf length}(G_{[n]})=\omega,$ since
$H_2(\eta_\omega)(G_{[n]})$ is divisible.
\end{Remark}

\section{Concluding remarks}

\begin{Lemma} For any prime $p$ and $n\geq 2$ the embedding $\mathbb Z_p \hookrightarrow \mathbb Q_p$ induces an isomorphism  
$$\Lambda^n \mathbb Z_p \cong \Lambda^n \mathbb Q_p.$$ In particular, $\Lambda^n \mathbb Z_p$ is a $\mathbb Q$-vector space of countable dimension.
\end{Lemma}
\begin{proof}
Since $\mathbb Z_p$ and $\mathbb Q_p$ are torsion free, the map $\mathbb Z_p^{\otimes n}\to \mathbb Q_p^{\otimes n} $ is a monomorphism. Since for torsion free groups the exterior power is embedded into the tensor power, we obtain that the map 
$\Lambda^n \mathbb Z_p \to \Lambda^n \mathbb Q_p$ is a monomorphism. So we can identify $\Lambda^n \mathbb Z_p$ with its image in $\Lambda^n \mathbb Q_p.$ Since $\mathbb Q_p=\bigcup \frac{1}{p^i} \cdot \mathbb Z_p,$ we obtain $\Lambda^n \mathbb Q_p=\bigcup \frac{1}{p^i} \cdot \Lambda^n  \mathbb Z_p.$ Then it is enough to prove that $\Lambda^n \mathbb Z_p$ is $p$-divisible. For any $a\in \mathbb Z_p$ we consider the  decomposition $a=a^{(0)}+a^{(1)},$ where $a^{(0)}\in \{0,\dots , p-1\}$ and $a^{(1)}=p\cdot b$ for some $b\in \mathbb Z_p.$  Then for any $a_1,\dots, a_n\in \mathbb Z_p$ we have $a_1 \wedge \dots \wedge a_n =\sum a_1^{(i_1)} \wedge \dots \wedge a_n^{(i_n)},$ where $(i_1,\dots ,i_n)$ runs over $\{0,1\}^n.$ For any sequence $(i_1,\dots,i_n)$ except $(0,\dots,0)$ the element $a_1^{(i_1)} \wedge \dots \wedge a_n^{(i_n)}$ is $p$ divisible because $a_i^{(1)}$ is $p$-divisible. Since $\Lambda^n \mathbb Z=0,$ we get $a_1^{(0)} \wedge \dots \wedge a_n^{(0)}=0.$ It follows that  $a_1 \wedge \dots \wedge a_n$ is $p$-divisible.  Thus $\Lambda^n \mathbb Z_p$ is $p$-divisible. 
\end{proof}

Remind that the Klein bottle group is given by $G_{Kl}=\mathbb Z
\rtimes C,$ where $C$ acts on $\mathbb Z$ by the multiplication on
$-1.$ Its pronilpotent completion is equal to $\mathbb Z_2 \rtimes C.$ Consider the map $w:\mathbb Z_2\times \mathbb Z_2 \to
\Lambda^2 \mathbb Z_2$ given by $w(a,b)=\frac{1}{2} \  a\wedge b.$ Here we use that $\Lambda^2 \mathbb Z_p$ is a $\mathbb Q$-vector space.  It is easy to
see that $w$ is a 2-cocycle and we get the corresponding central
extension
$$ \Lambda^2 \mathbb Z_2 \mono N_w \epi \mathbb Z_2  ,$$ whose underlying set is $(\Lambda^2 \mathbb Z_2) \times \mathbb Z_2$ and the product is given by
$$(\alpha,a)(\beta,b)=\left(\alpha+\beta+ \frac{1}{2} \cdot a\wedge b ,\  a+b\right).$$
We define the action of $C$ on $N_w$ by the map $(\alpha,a)\mapsto
(\alpha,-a).$

\begin{Proposition}\label{Proposition_localisation_of_Klein_bottle} There is an isomorphism  $$EG_{Kl}=N_w\rtimes C.$$ In other words, $EG_{Kl}$ can be described as the set $(\Lambda^2 \mathbb Z_2)
\times \mathbb Z_2 \times C$ with the multiplication given by
$$(\alpha,a,t^i)(\beta,b,t^j)=\left(\alpha+\beta+\frac{(-1)^i}{2}\cdot a\wedge b,\
a+(-1)^i b,\  t^{i+j}\right).$$
\end{Proposition}
\begin{proof} Set $G:=G_{Kl}.$ Then $\hat G=\mathbb Z_2 \rtimes C.$ The Lyndon-Hochschild-Serre spectral sequences imply that $H_2(G)=0$ and the
 map $\Lambda^2 \mathbb Z_2=H_2(\mathbb Z_2)\xrightarrow{\cong} H_2(\hat G)$ is an isomorphism.
Theorem \ref{Theorem_omega+1} implies that $T_{\omega+1}G=EG.$
 Then $EG$ is the universal relative central extension $(H_2(\eta_{\omega})\mono EG \epi \hat G,\eta_{\omega+1} ).$ Since
 $H_2(\eta_\omega)={\rm Coker}\{H_2(G)\to H_2(\hat G)\}$  and $H_2(G)=0,$ we obtain that $H_2(\mathbb Z_2) \xrightarrow{\cong}
 H_2(\hat G)\xrightarrow{\cong} H_2(\eta_\omega)$ are isomorphisms. The continious maps $B\mathbb Z_2 \to B\hat G \to {\sf Cone}(B\eta_\omega)$
 give the commutative diagram for any abelian group $A:$
\begin{equation}\label{diag_proof_Klein}
\xyma{ H^2(\eta_\omega,A)\ar@{->}[r]^\cong \ar@{->}[d] & {\sf
Hom}(H_2(\eta_\omega),A) \ar@{->}[d]^\cong
\\
H^2(\hat G,A) \ar@{->}[r]\ar@{->}[d] & {\sf Hom}(H_2(\hat G),A)
\ar@{->}[d]^\cong
\\
H^2(\mathbb Z_2,A)\ar@{->}[r]  & {\sf Hom}(H_2(\mathbb Z_2),A) }
\end{equation}
If $A$ is a divisible abelian group, then ${\sf Ext}(H_1(\hat
G),A)=0={\sf Ext}(\mathbb Z_2,A),$ and all morphisms in the
diagram \eqref{diag_proof_Klein} are isomorphisms. In particular,
if $A=H_2(\eta_\omega)=H_2(\hat G)=\Lambda^2 \mathbb Z_2$, then
the morphisms induce isomorphisms $$H^2(\eta_\omega,\Lambda^2
\mathbb Z_2)\cong H^2(\hat G,\Lambda^2 \mathbb Z_2) \cong
H^2(\mathbb Z_2,\Lambda^2 \mathbb Z_2) \cong {\sf Hom}(\Lambda^2
\mathbb Z_2,\Lambda^2 \mathbb Z_2)$$ and the extension $\Lambda^2
\mathbb Z_2\mono EG\epi \hat G$ corresponds to the identity map in
${\sf Hom}(\Lambda^2 \mathbb Z_2,\Lambda^2 \mathbb Z_2).$
Therefore, it is sufficient to prove that the extension $\Lambda^2
\mathbb Z_2 \mono N_w\rtimes C\epi \hat G$ goes to the identity
map via the composition $$H^2(\hat G,\Lambda^2 \mathbb Z_2)\to
H^2(\mathbb Z_2,\Lambda^2 \mathbb Z_2) \to {\sf Hom}(\Lambda^2
\mathbb Z_2,\Lambda^2 \mathbb Z_2).$$ The map $H^2(\hat
G,\Lambda^2 \mathbb Z_2)\to H^2(\mathbb Z_2,\Lambda^2 \mathbb
Z_2)$ on the level of extensions is given by the pullback. It
follows that the extension $\Lambda^2 \mathbb Z_2 \mono N_w\rtimes
C\epi \hat G$ goes to the extension $\Lambda^2 \mathbb Z_2 \mono
N_w\epi \mathbb Z_2.$ Then we need to prove that $w$ goes to the
identity map under the map $H^2(\mathbb Z_2,\Lambda^2 \mathbb
Z_2)\to {\sf Hom}(\Lambda^2 \mathbb Z_2, \Lambda^2 \mathbb Z_2).$

For any group $H$ and an abelian group $A$ the map $H^2(H,A)\to
{\sf Hom}(H_2(H),A)$ on the level of cocycles is induced  by the
evaluation map $C^2(H,A) \to {\sf Hom}(C_2(H),A)$ given by
$u\mapsto ((h, h') \mapsto u(h,h'))$. For an abelian group $H$ the
map $\Lambda^2 H \to H_2(H)$ given by $h \wedge  h'\mapsto (h,h')-(h',h)+B_2(H)$ is an
isomorphism $\Lambda^2 H\cong  H_2(H)$ (see 
\cite[Ch.V \S 5-6]{Brown}). Then the map $H^2(H,A)\to
{\sf Hom}(\Lambda^2 H,A)$ is given by $u\mapsto (h\wedge h'
\mapsto u(h,h')-u(h',h)).$ Since $w(a,b)-w(b,a)=\frac{1}{2} (a\wedge b-b\wedge a)=a\wedge b,$ $w$ goes to the identity map.
\end{proof}

\begin{Remark}
If we identify $EG_{Kl}$ with the set $(\Lambda^2 \mathbb Z_2
)\times \mathbb Z_2 \times C,$ it is easy to check that
$\gamma_2(EG_{Kl})=(\Lambda^2 \mathbb Z_2 )\times 2 \mathbb Z_2
\times 1$ and $[\gamma_2(EG_{Kl}),\gamma_2(EG_{Kl})]=(\Lambda^2
\mathbb Z_2 )\times 0 \times 1.$ It follows that $EG_{Kl}$ is a
solvable group of class 3 but it is    not metabelian. In
particular, the class of metabelian groups is not closed under the
$H\mathbb Z$-localization.
\end{Remark}

Finishing this paper, we mention some possibilities of
generalization of the obtained results. We conjecture that, the
tame $\mathbb Z[C]$-module $M$ defined by the $k\times k$-matrix
($k\geq 2$)
$$
\begin{pmatrix}
-1 & 1 & 0 & \dots & 0\\
 0 & -1 & 1 & 0 & \dots\\
 \ldots\\
 0 & \dots & 0 & 0 & -1
\end{pmatrix}
$$
defines the group $M\rtimes C$ with $H\mathbb Z$-length
$\omega+l(k),$ where $l(k)\in \mathbb N$ and $l(k)\to \infty$ for
$k\to \infty$. With the help of this group, one can try to use the
same scheme as in this paper to prove that $H\mathbb Z$-length of
a free non-cyclic group is $\geq 2\omega$.

\section*{Acknowledgement}
The research is supported by the Russian Science Foundation grant
N 16-11-10073.

\end{document}